%% file: Generation_Result-J3E.tex
\documentclass[10pt,reqno]{amsart} 
\usepackage{amsmath, amsthm, amssymb}  	
\usepackage{setspace}

\newtheorem{thm}{Theorem}[section]

\newtheorem{lem}[thm]{Lemma}
\newtheorem{prop}[thm]{Proposition}

\theoremstyle{remark} 
\newtheorem{rem}[thm]{Remark}

\theoremstyle{definition}

\numberwithin{equation}{section}

\newcommand{\dhook}{\stackrel{d}{\hookrightarrow}} % Densely embedded subspace.
\newcommand{\ddo}{\ddot{\text{o}}} % o with umlaught over it
\newcommand{\re}{\text{Re}} % Real part of complex number
\newcommand{\Tone}{\mathbb{T}}
\newcommand{\hAlphaN}{\left( h^{\alpha}(\Tone) \right)^n}
\newcommand{\dT}{d_{\mathbb{T}}}
\begin{document}

\title[Elliptic Operators and Maximal Regularity]
{Elliptic Operators and Maximal Regularity on Periodic little-H\"older Spaces}
%\date{\today}
\author[Jeremy LeCrone]{Jeremy LeCrone}
\address{Department of Mathematics,
         Vanderbilt University,
         Nashville, TN~37240, USA}
\email{jeremy.lecrone@Vanderbilt.Edu}

\begin{abstract}
We consider one-dimensional inhomogeneous parabolic equations with higher-order elliptic differential operators subject to 
periodic boundary conditions. In our main result we show that the property of continuous maximal regularity is
satisfied in the setting of periodic little-H$\ddo$lder spaces, provided the coefficients of the differential operator
satisfy minimal regularity assumptions.
% given only minimal regularity of coefficients. 
We address parameter-dependent elliptic equations, deriving invertibility and resolvent bounds which lead
to results on generation of analytic semigroups.
We also demonstrate that the techniques and results of the paper hold for elliptic differential
operators with operator-valued coefficients, in the setting of vector-valued functions.
%in the general setting
%of $E$-valued functions, for arbitrary Banach space $E$ over $\mathbb{C}$.
\end{abstract}

%################ End Preferences, Begin Document #####################

%\pagestyle{plain} % no headers or footers on the first page

\maketitle

%\doublespace

\input{PeriodicHolderSpaces-J3E}
\input{FourierMultipliers-J3E}
\input{GenerationWithConstantCoefficients-J3E}
\input{PerturbationResults-J3E}
\input{GenerationofAnalyticSemigroup-J3E}
\input{Conclusion-J3E}
\input{Vector-ValueCase-J3E}

\input{Bibliography-J3E}
\end{document}

%% file: PeriodicHolderSpaces-J3E.tex
\setcounter{section}{-1}
\section{Introduction}

In this paper we consider the following abstract periodic inhomogeneous equation
\begin{equation}\label{Eqn:CP}
\begin{cases}
\partial_t u(t,x) + \mathcal{A}(x,D)u(t,x) = f(t,x), & \text{$t > 0, \, x \in \mathbb{R}$}\\
u(0,x) = u_0(x), & \text{$x \in \mathbb{R}$},
\end{cases}
\end{equation}
where $\mathcal{A}(\cdot,D) := \sum_{k=0}^{2m} b_k(\cdot) D^k$ is a differential operator
of order $2m$, with variable coefficients $b_k: \mathbb{R} \rightarrow \mathbb{C}$. 
Further, we enforce periodic boundary conditions on the problem by assuming the given 
functions $u_0, b_k, f(t,\cdot)$, for $t \geq 0, \, k = 0, \ldots, 2m,$ are all $2 \pi$-periodic in $x \in \mathbb{R}$.
Hence, we will be looking for solutions $u(t,\cdot)$ which also exhibit $2 \pi$-periodicity on $\mathbb{R}$, for $t > 0$.
We will also consider the more general setting of 
vector-valued functions $u_0, f(t, \cdot), u(t, \cdot): \mathbb{R} \rightarrow E$,
and operator-valued coefficients $b_k: \mathbb{R} \rightarrow \mathcal{L}(E)$, 
for an arbitrary Banach space $E$ over $\mathbb{C}$. This more
general setting is discussed in Section~\ref{Section:VectorSetting}.

Understanding the nature of solutions (i.e. existence, uniqueness and
regularity) to inhomogeneous equations of this form is integral 
to the study of abstract quasilinear equations. In the quasilinear setting, 
we see that \eqref{Eqn:CP} takes the form
\[
\partial_t u + \texttt{A}(u,D) u = \texttt{F}(u),
\]
where the coefficients $b_k = b_k(u, u', \ldots, u^{(2m - 1)})$, and subsequently the 
differential operator $\texttt{A}$, may depend upon the solution $u$ and its
%we see that the 
%coefficients $b_k$ of $\mathcal{A}$ are dependent upon the solution $u$ and, possibly, the
lower order derivatives $u^{(j)}, \, j \leq 2m - 1$. Meanwhile, the inhomogeneity $f$ takes
the form $\texttt{F}(u) = \texttt{F}(u, u', \ldots, u^{(2m - 1)})$, for some nonlinear mapping $\texttt{F}$. 
%of order strictly less than $2m$. 
Several authors have studied abstract quasilinear equations, including \cite{Am93, AN90, CS01, KPW10, LUN95, Pru03}.  
Among the techniques employed to study quasilinear problems, 
the notion of maximal regularity has proven to be a valuable tool in establishing both
qualitative and quantitative results, c.f.
\cite{Am05, AN90, CS01, EMS98, KPW10, LUN95, Pru03, PSZ09, Sim95}. 
For a 
specific example of a quasilinear problem to which our results will apply, we reference the 
axisymmetric surface diffusion flow with periodic boundary conditions, 
as considered in \cite{BBW98}. 
In a forthcoming article, the results contained herein will be used to 
establish well-posedness, regularity and stability results for the 
periodic axisymmetric surface diffusion flow,
a one-dimensional fourth-order quasilinear problem.

Given an elliptic differential operator with periodic coefficients,
it is the goal of this paper to show that one can 
establish (continuous) maximal regularity results in the setting of periodic little-H$\ddo$lder spaces.
Moreover, we assume only minimal regularity conditions on the coefficients $b_k$, 
lending our results to applications in periodic quasilinear problems. In order to establish maximal 
regularity, we make use of a result originally proved by DaPrato and Grisvard \cite{DPG79},
which gives a construction of pairs of function spaces with
the property of continuous maximal regularity
for a given operator, under the assumption that the operator generates a strongly continuous analytic semigroup.
Hence, we focus first on showing generation of an analytic semigroup.

In fact, our results will show that elliptic operators with periodic coefficients 
generate analytic semigroups in the periodic H$\ddo$lder and little-H$\ddo$lder settings.
However, we focus on the results in the little-H$\ddo$lder space setting, because we get strong continuity
of the semigroups generated, due to density of embeddings in the little-H$\ddo$lder scale, a necessary condition
for applying the results of DaPrato and Grisvard.
To the best of the author's awareness, the work contained herein constitutes the first systematic treatment
of semigroup generation in the case of variable coefficients for elliptic operators with periodic boundary conditions. 
A related result for constant coefficients in the periodic setting was proved by Escher and Matioc \cite{EM08}, 
see also \cite{MAT09}, where
they considered a specific abstract operator of third order, in the periodic little-H$\ddo$lder setting.

In the process of establishing semigroup generation results, we consider the parameter-dependent 
elliptic equation
\[
(\lambda - \mathcal{A}(\cdot, D)) u = f, \qquad \lambda \in \mathbb{C},
\]
for which we show invertibility in the periodic H$\ddo$lder and little-H$\ddo$lder settings, provided 
$\re \, \lambda$ is sufficiently large. Additionally, we establish parameter-dependent estimates on
the resolvent of an elliptic operator under minimal regularity assumptions on the cefficients.
With invertibility and resolvent estimates, semigroup generation
follows from a standard result in semigroup theory, \cite{AM95} and \cite{FAT83}.
One will note that semigroup generation results are sufficient to derive well-posedness for the 
inhomogeneous problem \eqref{Eqn:CP} by classic semigroup techniques. However, as stated above, we focus on
establishing maximal regularity results, for which the little-H$\ddo$lder setting is desirable.

The paper is organized as follows. In the first section, we express regularity conditions for periodic functions on $\mathbb{R}$, exploiting a connection with functions defined over the one-dimensional torus $\Tone$, and establish necessary results regarding these spaces. In the second section we state and prove a Marcinkiewicz-type Fourier multiplier result, which is a slight generalization of a result in \cite{AB02}. In the third section we prove that a simplified operator $- \mathcal{A}_b$, with highest-order terms and constant coefficients, generates a (strongly continuous) analytic semigroup on periodic little-H$\ddo$lder spaces. In the fourth section, we extend this result to the principal part $-\mathcal{A}_p$, with highest-order terms and variable coefficients, using a partition technique as seen in \cite{AM01,AMHS94}. In the fifth section, we present a generation result for the full operator, discuss maximal regularity and solutions to the linear problem \eqref{Eqn:CP}. We conclude the paper by discussing the case of vector-valued functions and necessary modifications to our methods for results to carry over to this setting.

\section{Periodic Functions Over $\mathbb{R}$}\label{Section:PeriodicFunctions}

Given a $2 \pi$-periodic function $\tilde{f}: \mathbb{R} \rightarrow \mathbb{C}$ with some known regularity, we can restrict $\tilde{f}$ to an interval of periodicity (the interval $[-\pi, \pi]$, for instance) and the full function can still be recovered, i.e. the restricted function $f := \tilde{f}|_{[- \pi, \pi]}$ can be extended periodically to all of $\mathbb{R}$ and this extension will coincide exactly with $\tilde{f}$. Reversing this process, we want to start with a function $f: [ -\pi, \pi] \rightarrow \mathbb{C}$ and prescribe minimal conditions on $f$ so that the periodic extension exhibits desired regularity on $\mathbb{R}$. %To achieve this goal, we define a topology on the interval $[ -\pi, \pi]$ which works well with the periodic extension process.
In this section, we characterize several regularity classes for periodic functions with respect to their properties on the interval $[-\pi,\pi]$.

Let $\Tone := [- \pi, \pi]$, where the points $\pi$ and $-\pi$ are identified; we denote this point by $\{ \pi, - \pi \}$. Endow $\Tone$ with the metric topology $\tau$ generated by the metric
\[
\dT: \Tone \times \Tone \rightarrow \mathbb{R} \qquad \dT(x,y) := |x - y| \wedge (2\pi - |x - y|), \quad \text{where } a \wedge b := \min \{a,b\}.
\]
(For computational purposes, we follow the convention $\{ \pi, - \pi \} = \pi$ so that\\ $\dT (x, \{ \pi, - \pi \}) = | x - \pi | \wedge (2 \pi - |x - \pi|)$ and $\{ \pi, - \pi \} \geq x$ for all $x \in \Tone$.) Notice that $(\Tone, \tau)$ is a topological group which is isomorphic to the quotient group $\mathbb{R} / 2 \pi \mathbb{Z}$ endowed with the quotient topology. We see, moreover, that $\Tone$ is a complete, compact, metric space and we denote open balls of $\Tone$ by 
\[
\mathbb{B}_{\Tone}(x,\varepsilon) := \{ y \in \Tone : \dT(x,y) < \varepsilon \} \qquad \varepsilon > 0, \; x \in \Tone.
\]
We will demonstrate how using this metric gives intrinsic regularity conditions for functions 
defined on $\Tone$, which can be naturally extended periodically to $\mathbb{R}$. 

\subsection{Regularity on $\Tone$}\label{SubSec:Regularity}

Given a function $f: \Tone \rightarrow \mathbb{C}$, we define its periodic extension
\[
\tilde{f}(x) := f(x - 2 \pi k) \qquad \text{for} \quad x \in [\pi (2k - 1), \pi (2k + 1)], \quad k \in \mathbb{Z}.
\]
Denote by $\phi$ the periodic extension operator taking $f$ defined on $\Tone$ to $\tilde{f}$ defined on $\mathbb{R}$. One can immediately see that $\phi$ is bijective from $ \mathbb{C}^{\Tone}$ to $\left( \mathbb{C}^{\mathbb{R}} \right)_{per}$, the space of $2 \pi$-periodic functions on $\mathbb{R}$. Now, we define spaces of regular periodic functions over $\mathbb{R}$ into which we want $\phi$ to map.

Denote by $C_{per}(\mathbb{R})$ and $C_{per}^k(\mathbb{R})$ the spaces of $2 \pi$-periodic functions over $\mathbb{R}$ which are continuous and $k$-times continuously differentiable, respectively, for $k \in \mathbb{N}_0 := \mathbb{N} \cup \{ 0 \}$; we take $C^0_{per}(\mathbb{R}) = C_{per}(\mathbb{R})$ by convention. Each is a closed subspace of the corresponding non-periodic spaces and are Banach spaces when equipped with the following norms
\begin{equation}
\| f \|_{C(\mathbb{R})} := \sup_{x \in \mathbb{R}} |f(x)|, \qquad
\| f \|_{C^k(\mathbb{R})} := \sum_{j = 0}^k \| f^{(j)} \|_{C(\mathbb{R})}.
\end{equation}
Moreover, for $\alpha \in (0,1)$ and $k \in \mathbb{N}_0$, we define the space of H$\ddo$lder continuous functions $C^{k + \alpha}_{per}(\mathbb{R})$  to be those functions $f \in C^{k}_{per}(\mathbb{R})$ such that 
\[
[ f^{(k)} ]_{\alpha, \mathbb{R}} := \sup_{\substack{x,y \in \mathbb{R}\\ x \not= y}} \frac{|f^{(k)}(x) - f^{(k)}(y)|}{|x - y|^{\alpha}} < \infty.
\]
We call $[ \cdot ]_{\alpha, \mathbb{R}}$ the $\alpha$-H$\ddo$lder seminorm over $\mathbb{R}$ and one can see that $C^{k + \alpha}_{per}(\mathbb{R})$ is a Banach space with the norm
\begin{equation}\label{Eqn:HolderSpaceNorm}
\| f \|_{C^{k + \alpha}(\mathbb{R})} := \| f \|_{C^k(\mathbb{R})} + [ f^{(k)} ]_{\alpha, \mathbb{R}}.
\end{equation} 
For simplicity of notation, given $\theta \in \mathbb{R}_+$, we define $C^{\theta}_{per}(\mathbb{R}) := C^{\lfloor \theta \rfloor + \{ \theta \}}_{per}(\mathbb{R})$, where $\lfloor \theta \rfloor$ denotes the largest integer not exceeding $\theta$ and $\{ \theta \} := \theta - \lfloor \theta \rfloor$. 

With the periodic spaces over $\mathbb{R}$ established we define the spaces over $\Tone$ as follows. For $\theta \in \mathbb{R}_+$ let 
\begin{equation}
C^{\theta}(\Tone) := \left\{f \in \mathbb{C}^{\Tone} : \phi(f) \in C^{\theta}_{per}(\mathbb{R}) \right\} \qquad \text{with} \quad \| f \|_{C^{\theta}(\Tone)} := \| \phi(f) \|_{C^{\theta}(\mathbb{R})}.
\end{equation}
It follows immediately that $C^{\theta}(\Tone)$ is a Banach space and $\phi$ is a linear isometric isomorphism from $C^{\theta}(\Tone)$ to $C^{\theta}_{per}(\mathbb{R})$. Further, if $\theta \geq 1$ and $f \in C^{\theta}(\Tone)$, we define the derivative $f' \in \mathbb{C}^{\Tone}$ by $f' := \phi^{-1}( \phi(f)') = \left( \frac{d}{dx}\tilde{f} \right) \big|_{\Tone} \;$.

It is interesting to note that continuity, differentiability and H$\ddo$lder continuity can all be defined intrinsically on $\Tone$, making use of an ordered adaptation of the metric $\dT$, such that $\phi$ remains a linear isomorphism. Intrinsic definitions of regularity provide a different perspective for functions over the periodic domain $\Tone$ and setting for regularity independent of periodic extensions. Although the connection between functions over $\Tone$ and periodic functions over $\mathbb{R}$ has been used widely in the literature (c.f. \cite{AB04,EM08,ST87}), little attention has been paid to the local conditions and geometry on $\Tone$, which are important to the partition argument that we use in Section~\ref{Section:PartitionAndGeneration}. We will state some of the results regarding this intrinsic viewpoint that will be of use later in the paper, in particular we state equivalent definitions for (H$\ddo$lder) continuity over $\Tone$ and an application of the Mean Value theorem. For simplicity of notation, we denote by $\dT^{\alpha}(\cdot,\cdot)$ the quantity $\dT (\cdot,\cdot)^{\alpha}$.

\begin{prop}\label{Prop:IntrinsicDef} Let $f \in \mathbb{C}^{\Tone}$, then
\begin{enumerate}
	\item $f \in C(\Tone)$ if and only if $f$ is continuous in the metric topology $\tau$.
	\item for $\alpha \in (0, 1)$, $f \in C^{\alpha}(\Tone)$ if and only if $\displaystyle [ f ]_{\alpha, \Tone} := \sup_{\substack{x, y \in \Tone\\ x \not= y}} \frac{|f(x) - f(y)|}{\dT^{\alpha}(x,y)} < \infty.$\\ \vspace{.2em}
	Moreover, $[ f ]_{\alpha, \Tone} = [ \tilde{f} ]_{\alpha, \mathbb{R}}$ in this case.
	\item if $f \in C^1(\Tone)$ and $x, y \in \Tone$, then $|f(x) - f(y)| \leq \| f' \|_{C(\Tone)} \dT(x,y)$.
\end{enumerate}
\end{prop}

\begin{proof}
{\em a)} Follows from direct computation.

{\em b)} First, assume $f \in C^{\alpha}(\Tone)$ and let $x, y \in \Tone$ such that $x \not= y$, without loss of generality assume $x < y$ (recalling the convention the $\{ \pi, - \pi \} \geq x$ for all $x \in \Tone$). By definition of the metric $\dT$, we see that $\dT(x,y)$ is either equal to $|x-y|$ or $|(x + 2 \pi) - y|$. Now, we examine both cases
\begin{itemize}
	\item if $\dT(x,y) = |x - y|$, then $\displaystyle 	\frac{|f(x) - f(y)|}{\dT^{\alpha}(x,y)} = \frac{|\tilde{f}(x)-\tilde{f}(y)|}{|x - y|^{\alpha}} \leq [\tilde{f}]_{\alpha,\mathbb{R}}, $
	\item if $\dT(x,y) = |(x + 2 \pi) - y|$, then, by periodicity of $\tilde{f}$,
	\[
	\frac{|f(x) - f(y)|}{\dT^{\alpha}(x,y)} = \frac{|\tilde{f}(x)-\tilde{f}(y)|}{|(x + 2 \pi) - y|^{\alpha}} = \frac{|\tilde{f}(x + 2 \pi ) - \tilde{f}(y)|}{|(x + 2 \pi) - y|^{\alpha}} \leq [\tilde{f}]_{\alpha,\mathbb{R}}.
	\]
\end{itemize}
Hence, $[ f ]_{\alpha, \Tone} \leq [ \tilde{f} ]_{\alpha, \mathbb{R}} < \infty$. Conversely, assume that $[ f ]_{\alpha, \Tone} < \infty$ and consider $x, y \in \mathbb{R}$, $x < y$. Here we consider the following three cases
\begin{itemize}
	\item if $|x - y| \leq \pi$ and there exists $k \in \mathbb{Z}$ so that $x, y \in [\pi (2k - 1), \pi (2k + 1)]$. Then we have $(x - 2 \pi k), (y - 2 \pi k) \in \Tone, \; |x - y| = \dT((x - 2 \pi k), (y - 2 \pi k))$ and
	\[
	\frac{|\tilde{f}(x) - \tilde{f}(y)|}{|x - y|^{\alpha}} = \frac{|f(x - 2 \pi k) - f(y - 2 \pi k)|}{\dT^{\alpha}((x - 2 \pi k),(y - 2 \pi k))} \leq [ f ]_{\alpha, \Tone}.
	\]
	
	\item if $|x - y| \leq \pi$ and there exists $k \in \mathbb{Z}$ so that $x \in [\pi (2k - 1), \pi (2k + 1)]$ and $y \in (\pi (2k + 1), \pi (2k + 3)]$. Then we have that $(x - 2 \pi k), (y - 2 \pi (k + 1)) \in \Tone, \; |x - y| = \dT((x - 2 \pi k), (y - 2 \pi (k + 1)))$ and
	\[
	\frac{|\tilde{f}(x) - \tilde{f}(y)|}{|x - y|^{\alpha}} = \frac{|f(x - 2 \pi k) - f(y - 2 \pi (k + 1))|}{\dT^{\alpha}((x - 2 \pi k),(y - 2 \pi (k + 1) ))} \leq [ f ]_{\alpha, \Tone}.
	\]
	
	\item if $|x - y| > \pi$, then we can find $l \in \mathbb{Z}$ so that $|(x + 2 \pi l) - y| \leq \pi$. Then we have, taking advantage of the periodicity of $\tilde{f}$ and the fact that $|x - y| \geq |(x + 2 \pi l) - y|$, that
	\[
	\frac{|\tilde{f}(x) - \tilde{f}(y)|}{|x - y|^{\alpha}} \leq \frac{|\tilde{f}(x + 2 \pi l) - \tilde{f}(y)|}{|(x + 2 \pi l) - y|^{\alpha}} \leq [ f ]_{\alpha, \Tone},
	\]
	where the last inequality follows from the previous two cases.
\end{itemize}
Therefore, we can see that $[ \tilde{f} ]_{\alpha, \mathbb{R}} \leq [ f ]_{\alpha, \Tone} < \infty$, so $f \in C^{\alpha}(\Tone)$ and the claim follows. Moreover, we see that $[ f ]_{\alpha, \Tone} = [ \tilde{f} ]_{\alpha, \mathbb{R}}$.

{\em c)} Fix $f \in C^1(\Tone), \; x, y \in \Tone$ and assume, without loss of generality, that $x \leq y$. As before, it follows that $\dT(x,y)$ equals either $|x - y|$ or $|(x + 2 \pi) - y|$. We consider these two cases separately and see that the claim holds;
\begin{itemize}
	\item if $\dT(x,y) = |x - y|$, then  $\displaystyle |f(x) - f(y)| = |\tilde{f}(y) - \tilde{f}(x)| = \left| \int_x^y \tilde{f}'(t) dt \right|$\\
	$\displaystyle \leq \| \tilde{f}' \|_{C(\mathbb{R})}|y - x| = \| f' \|_{C(\Tone)} \dT(x,y).$
	\item if $\dT(x,y) = |(x + 2 \pi) - y|$, then  $\displaystyle	|f(x) - f(y)| = |\tilde{f}(x + 2 \pi) - \tilde{f}(y)| = \left| \int_y^{x + 2 \pi} \tilde{f}'(t) dt \right| \leq \| \tilde{f}' \|_{C(\mathbb{R})} |(x + 2 \pi) - y| = \| f' \|_{C(\Tone)} \dT(x,y).$
\end{itemize}
\end{proof}

Finally, we define the so-called little-H$\ddo$lder spaces over $\mathbb{R}$ and $\Tone$. We discuss equivalent characterizations and results on little-H$\ddo$lder spaces, important for maximal regularity and generation of analytic semigroups. For $\theta \in \mathbb{R}_+ \setminus \mathbb{Z}$ define the periodic little-H$\ddo$lder spaces over $\mathbb{R}$ as 
\begin{equation*}
h^{\theta}_{per}(\mathbb{R}) := \left\{ f \in C^{\theta}_{per}(\mathbb{R}) : \lim_{\delta \rightarrow 0} \sup_{\substack{x, y \in \mathbb{R}\\ 0 < |x - y| < \delta}} \frac{|f^{\lfloor \theta \rfloor}(x) - f^{\lfloor \theta \rfloor}(y)|}{|x - y|^{\{ \theta \}}} = 0 \right\}.
\end{equation*}
Then, $h^{\theta}_{per}(\mathbb{R})$ is a closed subspace of $C^{\theta}_{per}(\mathbb{R})$ and likewise a Banach space with the inherited norm $\| \cdot \|_{C^{\theta}(\mathbb{R})},$ defined by \eqref{Eqn:HolderSpaceNorm}. Moreover, it follows that the little-H$\ddo$lder spaces are, in fact, Banach algebras, in both the periodic and non-periodic settings. Now, we define $h^{\theta}(\Tone) := \{ f \in \mathbb{C}^{\Tone}: \phi(f) \in h^{\theta}_{per}(\mathbb{R}) \},$ for $\theta \in \mathbb{R}_+ \setminus \mathbb{Z}$. Following Proposition~\ref{Prop:IntrinsicDef}, one easily verifies that an equivalent definition is
\begin{equation}\label{IntrisicLittleHolder}
h^{\theta}(\Tone) := \left\{ f \in C^{\theta}(\Tone): \lim_{\delta \rightarrow 0} \sup_{\substack{x, y \in \Tone\\ 0 < \dT(x,y) < \delta}} \frac{|f^{\lfloor \theta \rfloor}(x) - f^{\lfloor \theta \rfloor}(y)|}{\dT^{\{ \theta \}}(x,y)} = 0 \right\}.
\end{equation}
Little-H$\ddo$lder spaces have been studied by several authors in context with analytic semigroups and maximal regularity, c.f. \cite{CS01,EM08,EMS98,LUN95}. The proposition that follows demonstrates two properties of little-H$\ddo$lder spaces which make them a natural choice for maximal regularity results. Before we state these results, we make a couple of comments on notation. 

If $E$ and $F$ are Banach spaces, we say that $E$ is continuously embedded in $F$, denoted $E \hookrightarrow F$, if there exists a continuous injective operator $i: E \rightarrow F$. Moreover, we say $E$ is densely embedded in $F$, denoted $E \dhook F$, if $i(E) \subset F$ is dense. Further, let $(\cdot, \cdot)_{\eta} := ( \cdot, \cdot)_{\eta, \infty}^0$ denote the continuous interpolation functor of Da Prato and Grisvard, with exponent $\eta \in (0,1)$, see \cite{AM95, LUN95} for reference. The following proposition is the periodic analog of well-known results on little-H$\ddo$lder spaces over $\mathbb{R}$, c.f. \cite{LUN95}.

\begin{prop}\label{Thm:hAlphaCAlphaClosure} \hfill
\begin{enumerate}
	\item For $\theta \in \mathbb{R}_+ \setminus \mathbb{Z}$ and $\sigma \in (\theta, \infty]$, $h^{\theta}(\Tone)$ is the closure of $C^{\sigma}(\Tone)$ in\\ $(C^{\theta}(\Tone), \| \cdot \|_{C^{\theta}(\Tone)})$. Hence, $h^{\sigma}(\Tone) \dhook h^{\theta}(\Tone)$ for $\sigma \in (\theta,\infty) \setminus \mathbb{Z}$. 
	\item For $\theta_1, \theta_2 \in \mathbb{R}_+ \setminus \mathbb{Z}$ with $\theta_2 \geq \theta_1$, it follows that\\ $(h^{\theta_1}(\Tone),h^{\theta_2}(\Tone))_{\eta} = h^{\eta \theta_2 + (1 - \eta)\theta_1}(\Tone)$, provided $(\eta \theta_2 + (1 - \eta)\theta_1) \notin \mathbb{Z}$.
\end{enumerate}
\end{prop}
\begin{proof}[Remarks on Proof:]
{\em a)} The proof of this statement is identical to the non-periodic case and can be found in Lunardi, \cite[Proposition 0.2.1]{LUN95}. We remark that the approximating functions from $C^{\infty}(\mathbb{R})$ established in Lunardi's proof, which are convolutions with smooth approximations of the identity, are in $C^{\infty}_{per}(\mathbb{R})$ in the periodic case. This fact follows from a property of convolutions involving periodic functions. Namely, given a function $\varphi$ and a $2 \pi$-periodic function $f$ such that the convolution $f \ast \varphi$ is well-defined, then the convolution is periodic, as
\[
(f \ast \varphi)(x + 2 \pi) = \int_{\mathbb{R}} f((x + 2 \pi) - y) \varphi(y) dy = \int_{\mathbb{R}} f(x - y) \varphi(y) dy = (f \ast \varphi)(x).
\]

{\em b)} The proof of this statement is identical to the non-periodic case, as demonstrated in \cite[Theorem 1.2.17]{LUN95}. Again, this method applies to the periodic case because we consider convolutions of smoothing kernels $\varphi_t$ with periodic functions $f$ over $\mathbb{R}$. Hence, the resulting convolutions are contained in $C^{\infty}_{per}(\mathbb{R})$.
\end{proof}
%%%%%%%%%%%%%%%%%%%%% BESOV SPACES

\subsection{Periodic Besov Spaces}

In order to state the Fourier multiplier theorem upon which our generation results heavily rely, we must first introduce the scale of Sobolev and Besov spaces. We present here a definition of periodic Besov spaces with respect to dyadic-type decompositions, similar to the development in \cite{AB04}, for more details on these spaces, and equivalent definitions, see Triebel and Schmeisser \cite[Section 3.5]{ST87}.

Following the notation of Arendt and Bu \cite{AB04}, let $\mathcal{D}(\Tone)$ denote the space $C^{\infty}(\Tone)$ equipped with the locally convex topology generated by the family of semi-norms $\| f \|_k := \| f^{(k)} \|_{C(\Tone)},$ for $k \in \mathbb{N}_0$. We define the space of periodic distributions $\mathcal{D}'(\Tone) := (\mathcal{D}(\Tone))^*$, the set of all bounded linear functionals on $\mathcal{D}(\Tone),$ and we equip $\mathcal{D}'(\Tone)$ with the weak-star topology over $\mathcal{D}(\Tone)$. Now we will investigate how the Fourier transform interacts with these spaces.

Denote by $e_k$ the function $[x \mapsto e^{ikx}]: \Tone \rightarrow \mathbb{C}$, then $e_k \in \mathcal{D}(\Tone)$ for $k \in \mathbb{Z}$. For $T \in \mathcal{D}'(\Tone)$, we define the \emph{Fourier coefficients} $\hat{T}(k) := \langle T, e_{-k} \rangle$, where $\langle \cdot, \cdot \rangle: \mathcal{D}'(\Tone) \times \mathcal{D}(\Tone) \rightarrow \mathbb{C}$ is the duality pairing. Notice that every test function $\varphi \in \mathcal{D}(\Tone)$ can be identified with the induced distribution $T_{\varphi} \in \mathcal{D}'(\Tone)$ defined by $\langle T_{\varphi}, \psi \rangle := \frac{1}{\sqrt{2 \pi}} \int_{- \pi}^{\pi} \varphi(x) \psi(x) dx, \, \psi \in \mathcal{D}(\Tone)$. Then the Fourier coefficients of $T_{\varphi}$ coincide with the usual Fourier coefficients for $\varphi \in \mathcal{D}(\Tone)$, namely
\[
\hat{T}_{\varphi}(k) = \hat{\varphi}(k) = \frac{1}{\sqrt{2 \pi}} \int_{-\pi}^{\pi} \varphi(x)e^{-ikx}dx.
\]
When no confusion is likely, we will denote by $\varphi$ both the function and its induced distribution. Moreover, by \cite[Theorem 12.5.3]{EDW82}, we have the Fourier series representation 
\[
f = \sum_{k \in \mathbb{Z}} \hat{f}(k) e_k  \qquad \text{for} \quad f \in \mathcal{D}'(\Tone) \quad \text{(convergence in $\mathcal{D}'(\Tone)$)}.
\]

To define Besov spaces over $\Tone$, let $\mathcal{S}(\mathbb{R})$ be the Schwartz space on $\mathbb{R}$ and $\mathcal{S}'(\mathbb{R})$ the space of tempered distributions on $\mathbb{R}$. Further, let $\Phi(\mathbb{R})$ denote the collection of all systems $(\varphi_j)_{j \in \mathbb{N}} \subset \mathcal{S}(\mathbb{R})$ satisfying the properties:
\begin{itemize}
	\item $\text{supp} \, \varphi_0 \subset [-2, 2], \qquad \text{supp} \, \varphi_j \subset [- 2^{j + 1}, - 2^{j - 1}] \cup [2^{j - 1}, 2^{j + 1}], \quad j \geq 1$,
	\item $\displaystyle \sum_{j \in \mathbb{N}_0} \varphi_j (x) = 1, \quad  x \in \mathbb{R}$,
	\item $\displaystyle \forall \, l \in \mathbb{N}_0, \; \exists \, C_l > 0 \text{ so that } \sup_{j \in \mathbb{N}_0} 2^{l j} \| \varphi_j^{(l)}\|_{C(\mathbb{R})} \leq C_l$.
\end{itemize}

Now, let $1 \leq p, q \leq \infty$, $s \in \mathbb{R}$ be fixed parameters and $\varphi = (\varphi_j) \in \Phi(\mathbb{R})$. For $f \in \mathcal{D}'(\Tone)$, $j \in \mathbb{N}_0$, the series $\sum_{k \in \mathbb{Z}} \varphi_j(k) \hat{f}(k) e_k$ has only finitely many nonzero terms, by compactness of the support of $\varphi_j$ (we refer to finite series of this form as \emph{trigonometric polynomials}), and it follows that $\sum_{k \in \mathbb{Z}} \varphi_j(k) \hat{f}(k) e_k \in L_p(\Tone)$. The norm on $L_p(\Tone)$ is given by
\[
\| g \|_p := \begin{cases} \displaystyle \left( \frac{1}{2 \pi} \int_{\Tone} |g(x)|^p dx \right)^{1/p} & \text{$1 \leq p < \infty$,}\\ \displaystyle \text{ess}\sup_{\hspace{-1.5em} x \in \Tone} |g(x)| & \text{$p = \infty$}. \end{cases}
\]
Now we define the periodic Besov space%, parametrized by $p, q, s$ and the system $\varphi$, as
\begin{equation}\label{Eqn:BesovSpaceDefined}
B^{s,\varphi}_{p,q}(\Tone) := \left\{ f \in \mathcal{D}'(\Tone) : \left( 2^{s j} \left\| \sum_{k \in \mathbb{Z}} \varphi_j(k) \hat{f}(k) e_k \right\|_p \right)_{j \in \mathbb{N}_0} \in \ell^q(\mathbb{N}_0) \right\}.
\end{equation}
Then $B^{s, \varphi}_{p,q}(\Tone)$ is a Banach space when equipped with the norm
\begin{equation}
\| f \|_{B^{s,\varphi}_{p,q}} := \begin{cases} \displaystyle \left( \sum_{j \in \mathbb{N}_0} 2^{s j q} \left\| \sum_{k \in \mathbb{Z}} \varphi_j(k) \hat{f}(k) e_k \right\|^q_p \right)^{1/q} & \text{for $q < \infty$,}\\ \displaystyle \sup_{j \in \mathbb{N}_0} 2^{s j} \left\| \sum_{k \in \mathbb{Z}} \varphi_j(k) \hat{f}(k) e_k \right\|_p & \text{for $q = \infty$.} \end{cases}
\end{equation}

Although the definition of a periodic Besov space depends explicitly upon the choice of system $\varphi \in \Phi(\mathbb{R})$, it can be shown that $(B^{s,\varphi}_{p,q}(\Tone), \| \cdot \|_{B^{s,\varphi}_{p,q}})$ is equivalent to $(B^{s,\psi}_{p,q}(\Tone), \| \cdot \|_{B^{s,\psi}_{p,q}})$, for two systems $\varphi, \psi \in \Phi(\mathbb{R})$, c.f. \cite[Theorem 3.5.1(i)]{ST87}. Hence, we drop reference to particular systems $\varphi \in \Phi(\mathbb{R})$ and simply refer to Besov spaces parametrized by $1 \leq p, q \leq \infty$ and $s \in \mathbb{R}$. See \cite{AM97,AM01,AB04,ST87} for more information on Besov spaces and their properties. We mention one property that comes up in the sequel, c.f. \cite[Theorem 3.1 (ii)]{AB04} or \cite[Theorem 3.5.4 (i)]{ST87}.

\begin{prop}\label{Thm:BisC}
For $s \in \mathbb{R}_+ \setminus \mathbb{Z}$, it holds that $B^s_{\infty, \infty}(\Tone) = C^s (\Tone)$.
\end{prop}

%% file: FourierMultipliers-J3E.tex
\section{A Fourier Multiplier Theorem}

The Fourier multiplier result that we will need is a slight modification of the result \cite[Theorem 4.5 (ii)]{AB04}, which gives sufficient conditions on the symbol of a Fourier multiplier so that the associated operator is continuous from $B^s_{p,q}(\Tone)$ to itself. We modify the result to get sufficient conditions for continuity from $B^s_{p,q}(\Tone)$ to $B^r_{p,q}(\Tone)$ for distinct values of $r$ and $s$. The modification we apply is the same technique used by B.V. Matioc \cite{MAT09} in altering the result \cite[Theorem 4.5 (i)]{AB04}. 

For $1 \leq p \leq \infty$, we define the Sobolev space $W^1_p(\Tone) := \{ f \in L_p(\Tone): f' \in L_p(\Tone) \}$ with the norm $\| f \|_{W^1_p} := \| f \|_p + \| f' \|_p.$

\begin{thm}\label{Thm:FourierMultipliers}
Let $r, s \in \mathbb{R}_+$ and $1 \leq p, q \leq \infty$. Suppose that $(M_k)_{k \in \mathbb{Z}} \subset \mathbb{C}$ is a sequence such that
\begin{equation}\label{Eqn:MultiplierAssumptions}
\begin{split}
s_1 &:= \sup_{k \in \mathbb{Z} \setminus \{ 0 \}} |k|^{r - s} | M_k | < \infty,\\
s_2 &:= \sup_{k \in \mathbb{Z} \setminus \{ 0 \}} |k|^{r - s + 1} | M_{k + 1} - M_k | < \infty.
\end{split}
\end{equation}
Then the Fourier multiplier with symbol $( M_k )_{k \in \mathbb{Z}}$ is a continuous mapping from $B^{s}_{p,q}(\Tone)$ to $B^{r}_{p,q}(\Tone)$, namely
\[
T: \left[ \sum_{k \in \mathbb{Z}} \hat{f}(k) e_k \longmapsto \sum_{k \in \mathbb{Z}} M_k \hat{f}(k) e_k \right] \in \mathcal{L} \left( B^{s}_{p,q}(\Tone), B^{r}_{p,q}(\Tone) \right).
\]
\end{thm}

The proof of this result relies upon the following Lemma, which is a simple version of \cite[Lemma 4.4]{AB04}. Here we only consider $\mathbb{C}$-valued functions over $\mathbb{R}$, so that the spaces involved are of Fourier type 2 and the statement is simplified as follows.

\begin{lem}\label{Lem:FourierTypeBound}
Let $1 \leq p \leq \infty$ and let $m \in C_c (\mathbb{R}, \mathbb{C}) \cap \mathcal{F}L_1 (\mathbb{R}, \mathbb{C})$. Then
\begin{equation}\label{Eqn:FourierTypeBound}
\left\| \sum_{k \in \mathbb{Z}} m(k) \hat{f}(k) e_k \right\|_p \leq C_p \eta_{2} (m) \| \sum_{k \in \mathbb{Z}} \hat{f}(k) e_k \|_p
\end{equation}
holds whenever $f \in L_p(\Tone)$ is a trigonometric polynomial, where $C_p$ is a constant depending only on $p$, and $\eta_{2}(m) := \inf \{ \| m(a \cdot) \|_{W^1_2}: a > 0 \}.$
\end{lem}

\begin{proof}[Proof of Theorem \ref{Thm:FourierMultipliers}:]
We provide the proof here for the reader's convenience and reference \cite[Theorem 4.5(ii)]{AB04} and \cite[Theorem 2.2.1]{MAT09}.
Fix $(M_k)_{k \in \mathbb{Z}} \subset \mathbb{C}$ satisfying \eqref{Eqn:MultiplierAssumptions} and parameters $s,r \in \mathbb{R}$, $1 \leq p, q \leq \infty$ and $\varphi := \{ \varphi_j \}_{j \geq 0} \in \Phi(\mathbb{R})$. We follow the same method as Arendt and Bu, with modifications to account for the (possibly nonzero) difference $|r-s|$, which is zero in the case considered in \cite{AB04}. To see that $T$ is a bounded operator from $B^{s}_{p,q}(\Tone)$ to $B^{r}_{p,q}(\Tone)$ as stated, it will suffice to show that there exists some constant $C > 0$ such that the bound
%\[
%2^{rj} \left\| \sum_{k \in \mathbb{Z}} M_k \varphi_j(k) \hat{f}(k) e_k \right\|_p \leq C 2^{sj} \left\|  \sum_{k \in \mathbb{Z}} \varphi_j(k) \hat{f}(k) e_k  \right\|_p
%\]
\[
\left\| \sum_{k \in \mathbb{Z}} \left( 2^{(r-s)j} M_k \right) \varphi_j(k) \hat{f}(k) e_k \right\|_p \leq C \left\|  \sum_{k \in \mathbb{Z}} \varphi_j(k) \hat{f}(k) e_k  \right\|_p,
\]
holds uniformly for $f \in B^s_{p,q}(\Tone)$ and $j \geq 0$. %Moreover, notice that it is equivalent to prove the uniform bound
\[
\left\| \sum_{k \in \mathbb{Z}} \left( 2^{(r-s)j} M_k \right) \varphi_j(k) \hat{f}(k) e_k \right\|_p \leq C \left\|  \sum_{k \in \mathbb{Z}} \varphi_j(k) \hat{f}(k) e_k  \right\|_p.
\]
To demonstrate this bound, we define an appropriate sequence of compactly supported continuous functions and take advantage of Lemma \ref{Lem:FourierTypeBound}.

For $j \geq 1$, define $m_j: \mathbb{R} \rightarrow \mathbb{C}$ by $m_j(x) = 0$ if $|x| \geq 2^{j + 2}$ or $|x| \leq 2^{j - 2}$, $m_j(k) = 2^{(r-s)j}M_k$ for $k \in \mathbb{Z}$ with $2^{j - 1} \leq |k| \leq 2^{j + 1}$, and $m_j$ is affine on $[k, k+1]$ for all $k \in \mathbb{Z}$.
%Moreover, $m_j$ is affine on the intervals $[-2^{j+2}, -2^{j+1}],$  $[-2^{j-1}, -2^{j-2}],$ $[2^{j-2}, 2^{j-1}],$ and  $[2^{j+1}, 2^{j+2}]$.
We define $m_0$ in a similar manner, where $m_0(x) = 0$ if $|x| \geq 2, \; m_0(k) = M_k$ for $-1 \leq |k| \leq 1$, and $m_0$ is affine on every interval $[k, k+1], \; k \in \mathbb{Z}.$

Now, one can see that $m_j \in C_c(\mathbb{R}) \cap \mathcal{F}L^1(\mathbb{R})$ and, by compactness of $\text{supp} \, \varphi_j$, $\sum_{k \in \mathbb{Z}} \varphi_j(k) \hat{f}(k) e_k$ is a trigonometric polynomial, for $j \geq 0$. Hence, we can apply Lemma \ref{Lem:FourierTypeBound} to see that, for $j \geq 1$, the following bounds hold.
%, where we realize $m_j(k) \in \mathbb{R}$ as an element of $\mathcal{L}(\mathbb{R})$ as a multiplication operator. Then, we have, for $j \geq 1$, taking advantage of the fact $\mathbb{R}$ has Fourier type $\rho = 2$,
\[
\begin{split}
&\left\| \sum_{k \in \mathbb{Z}} \left( 2^{(r - s)j} M_k \right) \varphi_j(k) \hat{f}(k) e_k \right\|_p = \left\| \sum_{2^{j - 1} \leq |k| \leq 2^{j + 1}} m_j(k) \varphi_j(k) \hat{f}(k) e_k \right\|_p\\
&\quad \leq C_p \, \eta_{_2}(m_j) \left\| \sum_{k \in \mathbb{Z}} \varphi_j(k) \hat{f}(k) e_k \right\|_p
\leq C_p \, \| m_j(2^j \cdot) \|_{W^1_2} \left\| \sum_{k \in \mathbb{Z}} \varphi_j(k) \hat{f}(k) e_k \right\|_p.
\end{split}
\]
Hence, it suffices to show that $\{ \| m_j(2^j \cdot ) \|_{W^1_2} \}_{j \geq 1}$ is uniformly bounded. From direct computation, one can see that this bound follows from the property $\text{supp} \, m_j \subset [ \frac{1}{4}, 4 ]$ and the bounds 
\[
\begin{split}
\sup_{x \in \mathbb{R}} |m_j(2^j x)| \leq \sup_{2^{j-1} \leq |k| \leq 2^{j+1}} 2^{(r - s)j} |M_k| \leq \sup_{2^{j-1} \leq |k| \leq 2^{j+1}} \left(  \frac{2^{(r-s)j}}{|k|^{r-s}} \right) s_1 \leq 2^{|r-s|} s_1, \\
\sup_{2^{j-1} \leq |p| \leq 2^{j+1}} 2^{(r-s+1)j}|M_{p+1} - M_p| \leq \sup_{2^{j-1} \leq |p| \leq 2^{j+1}} \left( \frac{2^{(r-s+1)j}}{|p|^{(r-s+1)}} \right) s_2 \leq 2^{|r-s+1|} s_2.
\end{split}
\]
Then, the $W_2^1(\Tone)$ norms can be bounded explicitly, for all $j \geq 0$, and it follows that the operator norm of $T$, as a bounded linear operator from $B^{s}_{p,q}(\Tone)$ to $B^{r}_{p,q}(\Tone)$, can be bounded in terms of the constants $s_1$ and $s_2$ alone.
\end{proof}

%% file: GenerationWithConstantCoefficients-J3E.tex
\section{Ellipticity and Generation of Analytic Semigroups}

Having established a setting within which we will look for solutions to the inhomogeneous Cauchy problem \eqref{Eqn:CP} in Section~\ref{Section:PeriodicFunctions}, we turn our attention back to the differential operator $\mathcal{A} = \mathcal{A}(\cdot, D)$. First, we define ellipticity conditions on a differential operator of order $2m$ and then we demonstrate our first result regarding generation of analytic semigroups on periodic little-H$\ddo$lder spaces. 

Denote by $D := i \frac{d}{dx}$ the elementary differential operator over $\Tone$ and let $m \in \mathbb{N}$ be an arbitrary positive integer. Now, fix a collection $\{ b_k : k = 0, \ldots, 2m \} \subset h^{\alpha}(\Tone)$ of coefficient functions and consider the differential operator $\mathcal{A}$, acting on $h^{2m + \alpha}(\Tone)$, defined by
\[
\mathcal{A}u(x) := \mathcal{A}(x,D) u(x) := \sum_{k = 0}^{2m} b_k(x) \, (D^k u)(x) = \sum_{k = 0}^{2m} i^k \, b_k(x) \, u^{(k)}(x), \qquad x \in \Tone.
\]
By the embedding property Proposition~\ref{Thm:hAlphaCAlphaClosure}(a) and the fact that $h^{\alpha}(\Tone)$ is a Banach algebra, it follows immediately that $\mathcal{A}$ maps $h^{2m + \alpha}(\Tone)$ into $h^{\alpha}(\Tone)$. Now, denote by $\sigma \mathcal{A}: \Tone \times \mathbb{R} \rightarrow \mathbb{C}$ the \emph{principal symbol} of $\mathcal{A}$, defined by $\sigma \mathcal{A}(x,\xi) :=  b_{2m}(x) \xi^{2m}.$ Then we say that $\mathcal{A}$ is a \emph{uniformly elliptic} operator on $\Tone$ if there exists a constant $c_1 > 0$ such that 
\begin{equation}\label{Eqn:UniformEllipticity}
\re \, \big( \sigma \mathcal{A}(x,\xi) \big) \geq c_1 \qquad \text{for all} \quad x \in \Tone, \; | \xi | = 1.
\end{equation}
In case $b_{2m}$ is simply a $\mathbb{R}$-valued function, we see that uniform ellipticity is equivalent to the condition $b_{2m}(x) \geq c_1$ for all $x \in \Tone$. Meanwhile, when $b_{2m}$ takes values in $\mathbb{C} \setminus \mathbb{R}$, uniform ellipticity is equivalent to the more general condition $b_{2m}(\Tone) \subset \{ z \in \mathbb{C} : \re \, z \geq c_1 \}.$ Also notice, by assumption we have $b_{2m}$ continuous on $\Tone$, so that there always exists some constant $c_2 > 0$ for which $\| b_{2m} \|_{C(\Tone)} \leq c_2$. 

Following the notation of Amann \cite{AM95}, given Banach spaces $E_0$ and $E_1$ with $E_1 \dhook E_0$, we denote by $\mathcal{H}(E_1,E_0)$ the collection of $A \in \mathcal{L}(E_1,E_0)$ such that $-A$ is the infinitesimal generator of an analytic semigroup on $E_0$, with domain $D(A) = E_1$. Moreover, given parameters $\kappa \geq 1$, $\omega > 0$, we denote by $\mathcal{H}(E_1, E_0, \kappa, \omega)$ the set of linear operators $A: E_1 \rightarrow E_0$, closed in $E_0$, such that $\omega + A \in \mathcal{L}_{isom}(E_1, E_0)$ and 
\[
\kappa^{-1} \leq \frac{\|(\lambda + A)x\|_0}{|\lambda| \|x\|_0 + \|x\|_1} \leq \kappa, \qquad x \in E_1 \setminus \{ 0 \}, \, \re \, \lambda \geq \omega.
\]
Then, it follows, c.f. \cite[Theorem 1.2.2]{AM95}, that $\displaystyle \mathcal{H}(E_1,E_0) = \bigcup_{\substack{\kappa \geq 1\\ \omega > 0}} \mathcal{H}(E_1,E_0,\kappa,\omega)$.

\begin{thm}\label{Thm:GenerationWithConstant}
Let $m \in \mathbb{N}$, $\alpha \in \mathbb{R}_+ \setminus \mathbb{Z}$ and consider the differential operator $\mathcal{A}_b := b \,  D^{2m}$ with constant coefficient $b \in \mathbb{C}$. If $\mathcal{A}_b$ is uniformly elliptic, with constant $c_1 > 0$,
%, i.e. there exist constants $c_1, c_2 \in \mathbb{R}_+ \setminus \{ 0 \}$ so with $\re \, b \geq c_1 > 0$ 
and $c_2 \geq c_1 > 0$ is chosen so that $|b| \leq c_2$, then $- \mathcal{A}_b$ generates a (strongly continuous) analytic semigroup on $h^{\alpha}(\Tone)$ with domain $h^{2m + \alpha}(\Tone)$. Moreover, for any $\omega > 0$, there exists $\kappa = \kappa(\omega, c_1, c_2, m)$ such that 
\[
\mathcal{A}_b \in \mathcal{H} \big(h^{2m + \alpha},h^{\alpha},\kappa(\omega, c_1, c_2, m), \omega \big).
\]
\end{thm}

The method for proving this theorem is inspired by an argument presented by Escher and Matioc in \cite{EM08}, where they demonstrated that a particular third order operator, associated with Stokesian Hele-Shaw flow, generates an analytic semigroup on periodic little-H$\ddo$lder spaces.  Before we present the proof, we need to state a result which helps establish a connection between little-H$\ddo$lder spaces and Fourier multiplier results, stated earlier in the scale of Besov spaces. First, if we can apply Theorem~\ref{Thm:FourierMultipliers} for the case $p = q = \infty$, then the identification in Proposition~\ref{Thm:BisC} gives results on $C^s(\Tone)$ which we then need to connect with the little-H$\ddo$lder spaces $h^s(\Tone)$. 

\begin{lem}\label{Lem:Ctoh}
Suppose $T \in \mathcal{L}(C^{k + \alpha}(\Tone), C^{l + \alpha}(\Tone))$ such that\\ $T (C^{k + r}(\Tone)) \subset C^{l + r}(\Tone)$, for $k,l \in \mathbb{N}_0, \, \alpha \in \mathbb{R}_+ \setminus \mathbb{Z}$ and $r > \alpha$. Then $T \in \mathcal{L}(h^{k + \alpha}(\Tone), h^{l + \alpha}(\Tone))$.
\end{lem}
\begin{proof}
This result is a straight forward consequence of the dense embedding\\ $C^{l + r}(\Tone) \dhook h^{l + \alpha}(\Tone)$, c.f. Proposition~\ref{Thm:hAlphaCAlphaClosure}(a), we present the proof here for the readers convenience. First, notice that for $T \in \mathcal{L}(C^{k + \alpha}(\Tone), C^{l + \alpha}(\Tone))$, it follows that $T \in \mathcal{L}(h^{k + \alpha}(\Tone),C^{l + \alpha}(\Tone))$. Hence, it suffices to show that $T(h^{k + \alpha}(\Tone)) \subset h^{l + \alpha}(\Tone)$. Let $f \in h^{k + \alpha}(\Tone)$ and we can find $(f_j)_j \subset C^{k + r}(\Tone)$ such that $f_j \rightarrow f$ in $\| \cdot \|_{C^{k + \alpha}}$. Then $Tf_j \rightarrow Tf$ in $\| \cdot \|_{C^{l + \alpha}},$ by $T \in \mathcal{L}(C^{k + \alpha}(\Tone), C^{l + \alpha}(\Tone))$, and $Tf_j \in C^{l + r}(\Tone)$ for $j \in \mathbb{N}$, by assumption. Therefore, we have $Tf \in \overline{C^{l + r}(\Tone)}^{\| \cdot \|_{C^{l + \alpha}}} = h^{l + \alpha}(\Tone)$ and the lemma is proved.
\end{proof}

\begin{proof}[Proof of Theorem \ref{Thm:GenerationWithConstant}]
Fix $\alpha \in \mathbb{R}_+ \setminus \mathbb{Z}$, $\omega > 0$ and $b \in \mathbb{C}$ as indicated, in particular we assume that $b \in \Sigma \, (c_1,c_2) := \{ z \in \mathbb{C}: \re \, z \geq c_1 \} \cap \{ z \in \mathbb{C}: | z | \leq c_2 \}$. First, we realize the operator $- \mathcal{A}_b$ as a Fourier multiplier. Since
\[
\mathcal{A}_b \left( \sum_{k \in \mathbb{Z}} a_k e_k \right) 
= \sum_{k \in \mathbb{Z}}  b (i)^{2m} a_k (i k)^{2m} e_k = \sum_{k \in \mathbb{Z}}  b k^{2m} a_k e_k,
\]
we see that $-\mathcal{A}_b$ is associated with the multiplier symbol $(M_k)_k := (- b k^{2m})_k.$ 

\noindent {\bf Claim 1:} \emph{$(\lambda + \mathcal{A}_b) \in \mathcal{L}_{isom} (h^{2m + \alpha}(\Tone), h^{\alpha}(\Tone))$ for $\re \, \lambda \geq \omega$, i.e. 
	\[
	\rho(-\mathcal{A}_b) \supset \{ \lambda \in \mathbb{C}: \re \, \lambda \geq \omega \}.
	\]
	Moreover, the set $\{ \| (\lambda + \mathcal{A}_b)^{-1} \|_{\mathcal{L}(h^{\alpha},h^{2m + \alpha})}: \re \, \lambda \geq \omega \}$ is uniformly bounded by some $M_1 = M_1(\omega, c_1, c_2, m) < \infty$.}
	
First notice that $(\lambda + \mathcal{A}_b) \in \mathcal{L}(C^{2m + \sigma}(\Tone),C^{\sigma}(\Tone))$ is a natural consequence of the embedding $C^{2m + \sigma}(\Tone) \hookrightarrow C^{\sigma}(\Tone)$, for arbitrary $\sigma \in \mathbb{R}_+$. In particular, we see that
\[
\| (\lambda + \mathcal{A}_b)f \|_{C^{\sigma}} \leq |\lambda| \| f \|_{C^{\sigma}} + |b| \, \| f^{(2m)} \|_{C^{\sigma}} \leq (c(\sigma) \, |\lambda|  + c_2 )\| f \|_{C^{2m + \sigma}},
\]
where $c(\sigma) > 0$ is the embedding constant, i.e. $\| f \|_{C^{\sigma}} \leq c(\sigma) \, \| f \|_{C^{2m + \sigma}}$ for all $f \in C^{2m + \sigma}(\Tone)$. Now, we focus on showing continuous invertibility of the operator $(\lambda + \mathcal{A}_b)$. We will demonstrate invertibility in the classic H$\ddo$lder spaces, then apply Lemma~\ref{Lem:Ctoh} to get the stated result.

We use Theorem \ref{Thm:FourierMultipliers} and the identification $B^{\sigma}_{\infty,\infty}(\Tone) = C^{\sigma}(\Tone),$ for $\sigma \in \mathbb{R}_+ \setminus \mathbb{Z}$. In particular, let $\re \, \lambda \geq \omega$ and consider the symbol $\left( \tilde{M}_k(\lambda) \right)_k := \left( \frac{1}{\lambda + bk^{2m}} \right)_k,$ which we will show  
satisfies \eqref{Eqn:MultiplierAssumptions}, with $r = 2 m + \sigma$ and $s = \sigma$. Then $r - s = 2 m$ and we have,
\begin{align*}
	|k|^{2m}|\tilde{M}_k(\lambda)| = \frac{k^{2m}}{|\lambda + b k^{2m}|} & \leq \frac{k^{2m}}{\re \, b \, k^{2m}} \leq \frac{1}{\re \,b} \quad \text{for } k \in \mathbb{Z} \setminus \{ 0 \}\\
	\Longrightarrow \qquad  s_1 := \sup_{k \in \mathbb{Z} \setminus \{ 0 \}}& |k|^{r - s} |\tilde{M}_k(\lambda)| \leq \frac{1}{c_1} < \infty,
\end{align*}
and
\begin{align*}
	|k|^{2m + 1}|\tilde{M}_{k+1}(\lambda) - \tilde{M}_k(\lambda)| &= |k|^{2m + 1} \left| \frac{1}{\lambda + b (k + 1)^{2m}} - \frac{1}{\lambda + bk^{2m}} \right| \\
	%&= |k|^{2m + 1} \left| \frac{b ((k + 1)^{2m} - k^{2m})}{(\lambda + b(k + 1)^{2m})( \lambda + bk^{2m} )} \right|\\
	&= \frac{|k|^{2m}}{|\lambda + b(k + 1)^{2m}|} \frac{|k|^{2m}}{|\lambda + bk^{2m}|} \frac{|b| |(k+1)^{2m} - k^{2m}|}{|k|^{2m - 1}}\\
	&\leq \frac{|k|^{2m}}{|\lambda + b(k + 1)^{2m}|} \frac{|b|}{\re \, b} \frac{|(k+1)^{2m} - k^{2m}|}{|k|^{2m - 1}}
\end{align*}
	%&\leq \begin{cases}
	%\frac{1}{\omega} & \text{for $k = -1$}\\
	%\frac{1}{b} \, \frac{k^{2m}}{(k + 1)^{2m}} \, \frac{|(k+1)^{2m} - k^{2m}|}{|k|^{2m - 1}} & \text{for $k \not= 0, -1$}
	%\end{cases}\\
If $k = -1$, then this last term is equal $\frac{1}{|\lambda|}$, which is majorized by $\frac{1}{\omega}$. For all other $k \in \mathbb{Z} \setminus \{0 \}$, we eliminate dependence on $\lambda$, as in claim 1, so that we have %this last term majorized by 
%\[
%\frac{|b|}{\re \, b} \frac{|k|^{2m}}{\re \, b \, (k + 1)^{2m}} \frac{|(k+1)^{2m} - k^{2m}|}{|k|^{2m - 1}} \leq \frac{|b|}{(\re \, b)^2} \frac{k^{2m}}{(k + 1)^{2m}} \sum_{j=0}^{2m-1} \binom{2m}{j} k^{j - 2m + 1}
%\]
\begin{align*}
	 s_2 := \sup_{k \in \mathbb{Z} \setminus \{ 0 \}}& |k|^{r - s + 1} \Big| \tilde{M}_{k + 1}(\lambda) - \tilde{M}_k(\lambda) \Big|\\
	&\hspace{-4.5em}\leq \bigg(\frac{1}{\omega} \vee \frac{c_2}{(c_1)^2} \bigg) \, \sup_{k \in \mathbb{Z} \setminus \{ -1 \}} \Bigg( \frac{|k|^{2m}}{|k + 1|^{2m}} \sum_{j=0}^{2m-1} \binom{2m}{j} |k|^{j - 2m + 1} \Bigg) < \infty.
\end{align*}
Hence, by Theorem~\ref{Thm:FourierMultipliers} we have $R(\lambda) \in \mathcal{L}(B^{r}_{p,q}(\Tone), B^{r + 2m}_{p,q}(\Tone))$ for any $1 \leq p, q \leq \infty$ and $r \in \mathbb{R}_+$, where $R(\lambda)$ is the operator associated with the symbol $\left( \tilde{M}_k(\lambda) \right)_k$. Taking $p = q = \infty$ and $r = \sigma$, we see $R(\lambda) \in \mathcal{L}(C^{\sigma}(\Tone), C^{2m + \sigma}(\Tone))$. Meanwhile, it holds that
\[
R(\lambda)(\lambda + \mathcal{A}_b)f = f \quad \text{and} \quad (\lambda + \mathcal{A}_b)R(\lambda)g = g, \quad \text{for} \quad f \in C^{2m + \sigma}(\Tone), \; g \in C^{\sigma}(\Tone),
\]
which demonstrates that $R(\lambda) = (\lambda + \mathcal{A}_b)^{-1}$ and
\begin{equation}
(\lambda + \mathcal{A}_b) \in \mathcal{L}_{isom}(C^{2m + \sigma}(\Tone), C^{\sigma}(\Tone)) \qquad  \text{for} \quad  \re \, \lambda \geq \omega, \, \sigma \in \mathbb{R}_+ \setminus \mathbb{Z}.
\end{equation} 
Now, it is clear that Claim 1 follows from Lemma \ref{Lem:Ctoh}.

Meanwhile, for any $\re \, \lambda \geq \omega$, notice that $s_1$, $s_2$ and so, by Theorem~\ref{Thm:FourierMultipliers}, the operator norm $\| (\lambda + \mathcal{A}_b)^{-1} \|_{\mathcal{L}(h^{\alpha},h^{2m + \alpha})}$ can be bounded by terms depending only on the constants $\omega, c_1, c_2$ and $m$.
In particular, there exists some $M_1 = M_1(\omega, c_1, c_2, m) < \infty$ such that $\| (\lambda + \mathcal{A}_b)^{-1} \|_{\mathcal{L}(h^{\alpha},h^{2m + \alpha})} \leq M_1$ for all $\re \, \lambda \geq \omega$ and for all $b \in \Sigma \, (c_1,c_2)$.

\noindent {\bf Claim 2:} \emph{$\lambda (\lambda + \mathcal{A}_b)^{-1} \in \mathcal{L}(h^{\alpha}(\Tone))$ for $\re \, \lambda \geq \omega,$ and there is an upper bound $M_2 = M_2(\omega, c_1, c_2, m) < \infty$ for the set $\left\{ | \lambda | \|(\lambda + \mathcal{A}_b)^{-1} \|_{\mathcal{L}(h^{\alpha}(\Tone))} : \re \, \lambda \geq \omega \right\}$.}

Fix $\re \, \lambda \geq \omega$ and notice that the operator $\lambda (\lambda + \mathcal{A}_b)^{-1}$ has the associated multiplier symbol $\left( \frac{ \lambda }{\lambda + b k^{2m}} \right)_k$. We established in Claim 1 that $(\lambda + \mathcal{A}_b)^{-1}$ is a well-defined operator mapping $h^{\alpha}(\Tone)$ into $h^{2m + \alpha}(\Tone)$. Now, by the embedding property, Theorem \ref{Thm:hAlphaCAlphaClosure}(a), we can also consider the mapping properties of $(\lambda + \mathcal{A}_b)^{-1}$ as an operator from $h^{\alpha}(\Tone)$ into itself. Again, we make use of Lemma \ref{Lem:Ctoh} and Theorem \ref{Thm:FourierMultipliers}, where now we are taking $r = s = \sigma$ and $p = q = \infty$. Moreover, we show that $s_1$ and $s_2$ can be bounded independent of $\re \, \lambda \geq \omega$.
% and the operator norms$|\lambda| \| (\lambda + \mathcal{A}_b)^{-1} \|_{\mathcal{L}(h^{\alpha}(\Tone))}$ can be bounded above by a constant $M_2 = M_2(c_1, c_2, \omega)$. 

Notice that we can find $\vartheta = \vartheta(c_1, c_2) \in (0,\frac{\pi}{2})$ such that $\Sigma \, (c_1,c_2) := \{ z \in \mathbb{C}: \re \, z \geq c_1 \} \cap \{ z \in \mathbb{C}: |z| \leq c_2 \} \subset S_{\vartheta} := \{ z \in \mathbb{C}: |\text{arg} \, z| < \vartheta \}$. %, where we call $S_{\vartheta}$ the \emph{sector} of $\mathbb{C}$ of angle $\vartheta$. 
Moreover, there exists a constant $C(\vartheta)$ such that $|\lambda + z| \geq |\lambda| / C(\vartheta) $ for all $z \in S_{\vartheta} \cup \{ 0 \}$, $\re \, \lambda > 0$, since $\vartheta < \frac{\pi}{2}$. In particular, we have 
\[
s_1 = \sup_{k \in \mathbb{Z} \setminus \{ 0 \}} \frac{|\lambda|}{|\lambda + b \, k^{2m}|} \leq C(\vartheta) \qquad \text{for all} \quad \re \, \lambda \geq \omega.
\] 
Now, considering $s_2$, we have the bound
\begin{align*}
|k| \bigg|& \frac{|\lambda|}{ \lambda + b (k + 1)^{2m}} - \frac{|\lambda|}{ \lambda + b k^{2m}}\bigg|\\
&= \frac{|\lambda|}{|\lambda + b (k+1)^{2m}|} \frac{k^{2m}}{| \lambda + bk^{2m}|}\frac{|b||(k + 1)^{2m} - k^{2m}|}{|k|^{2m - 1}}\\
&\leq C(\vartheta) \, \frac{k^{2m}}{\re \, b \, k^{2m}} \, \frac{|b|((k+1)^{2m} - k^{2m})}{|k|^{2m - 1}}\\ 
&\leq C(\vartheta) \, \frac{c_2}{c_1} \, \frac{(k+1)^{2m} - k^{2m}}{|k|^{2m - 1}},
\end{align*}
for $k \in \mathbb{Z} \setminus \{ 0 \}$. Hence, 
\[
s_2 \leq \left( C(\vartheta) \, \frac{c_2}{c_1} \right)  \, \sup_{k \in \mathbb{N}} \left( \sum_{j=0}^{2m-1} \binom{2m}{j} k^{j - 2m + 1} \right) < \infty,
\]
again uniformly in $\re \, \lambda \geq \omega$. Now we see that $\lambda (\lambda + \mathcal{A}_b)^{-1} \in \mathcal{L}(C^{\sigma}(\Tone), C^{2m + \sigma}(\Tone))$ holds by application of Theorem \ref{Thm:FourierMultipliers}, for $\lambda \geq \omega$, $\sigma \in \mathbb{R}_+ \setminus \mathbb{Z}$. Hence, the claim holds by Lemma~\ref{Lem:Ctoh} and we fix a constant $M_2 = M_2(\omega, c_1, c_2, m) < \infty$ such that $\|(\lambda + \mathcal{A}_b)^{-1}\|_{\mathcal{L}(h^{\alpha})} \leq M_2 / |\lambda|$ holds uniformly for $\re \, \lambda \geq \omega$ and $b \in \! \Sigma \, (c_1,c_2)$.

By claims 1 and 2, we see that $- \mathcal{A}_b$ satisfies the conditions necessary to generate an analytic semigroup, see Amann \cite[Theorem 1.2.2]{AM95} for instance. Moreover, if we choose $\kappa = \kappa(c_1, c_2, \omega) \geq 2 \big(M_1 \vee M_2 \big) \vee \big( 1 \vee c_2 \big)$ it holds that
\[
\kappa^{-1} \leq \frac{\| (\lambda + \mathcal{A}_b) f \|_{h^{\alpha}(\Tone)}}{|\lambda|\| f \|_{h^{\alpha}(\Tone)} + \| f \|_{h^{2m + \alpha}(\Tone)}} \leq \kappa, \qquad f \in h^{2m + \alpha}(\Tone) \setminus \{ 0 \}, \; \re \, \lambda \geq \omega. 
\]
Hence, we see that $\mathcal{A}_b \in \mathcal{H}(h^{\alpha}(\Tone),h^{2m + \alpha}(\Tone), \kappa, \omega)$, as claimed.
%\[
%|\lambda|\| f \|_{h^{\alpha}(\Tone)} \leq \kappa \|(\lambda + \mathcal{A}_b)f \|_{h^{\alpha}(\Tone)}, \qquad f \in h^{2m + \alpha}(\Tone), \; \re \, \lambda \geq \omega.
%\]
%Hence, by \cite[Remarks 1.2.1 (a)]{AM95}, we see that $\mathcal{A}_b \in \mathcal{H}(h^{\alpha}(\Tone),h^{2m + \alpha}(\Tone), \kappa_1, \omega)$, where
%\[
%\kappa_1 := \| \mathcal{A}_b \|_{\mathcal{L}(h^{2m + \alpha}(\Tone),h^{\alpha}(\Tone))} \vee \kappa \left( 1 + 3 \| (\omega + \mathcal{A}_b)^{-1} \|_{\mathcal{L}(h^{\alpha}(\Tone), h^{2m + \alpha}(\Tone))} \right).
%\]
%Notice that $\| \mathcal{A}_b \|_{\mathcal{L}(h^{2m + \alpha}(\Tone),h^{\alpha}(\Tone))}$ is majorized by terms depending on $b \leq c_2$, and \\$\| (\omega + \mathcal{A}_b)^{-1} \|_{\mathcal{L}(h^{\alpha}(\Tone), h^{2m + \alpha}(\Tone))}$ is majorized by terms depending on $\frac{1}{b} \leq \frac{1}{c_1}$. Hence, we see that $\mathcal{A}_b \in \mathcal{H}(h^{\alpha}(\Tone), h^{2m + \alpha}(\Tone), \kappa(c_1, c_2, \omega), \omega)$ with $\kappa(c_1, c_2, \omega)$ growing like $\frac{1}{\omega}, \, \frac{1}{c_1}$ and $c_2$.

\end{proof}

%\begin{rem}
%The definition that we are using for uniform ellipticity can be generalized slightly without affecting the method used in the proof of Theorem~\ref{Thm:GenerationWithConstant}. In particular, following the terminology of \cite{AM01} and \cite{AMHS94}, if the operator $\mathcal{A}$ is  \emph{uniformly $(M, \theta_0)$-elliptic} for some $M > 0$ and $\theta_0 < \frac{\pi}{2}$, then the methods of the proof still hold, where we find that the angle $\vartheta$, found in Claim 3, can be taken to be $\theta_0$ in this definition. We use the provided definition of uniform ellipticity in this paper, because the conditions required on the leading coefficient $b_{2m}$ are easier to verify in applications than those for the $(M,\theta_0)$-ellipticity definition.
%\end{rem}

%% file: PerturbationResults-J3E.tex
\section{Partition and Generation Result}\label{Section:PartitionAndGeneration}

Now that we have generation results for the operator with constant coefficients, we can extend these results to variable coefficients through the following partition and perturbation argument. Here we consider the operator
\begin{equation}\label{Eqn:OperatorDefined}
\mathcal{A}_p := \mathcal{A}_p(\cdot,D) := b(\cdot)D^{2m}, \qquad \text{for} \quad b \in \mathbb{C}^{\Tone},
\end{equation}
and we assume that $\mathcal{A}_p$ satisfies the conditions of uniform ellipticity \eqref{Eqn:UniformEllipticity}.
We will show that, under minimal regularity assumptions on the coefficient function $b$, 
$\mathcal{A}_p$ generates an analytic semigroup on $h^{\alpha}(\Tone)$ with domain $h^{2m + \alpha}(\Tone)$. %, $\alpha \in \mathbb{R}_+ \setminus \mathbb{Z}$.

For the following localization argument, we make use of the fact that $\Tone$ is isomorphic to the (additive) quotient group $\mathbb{R} / 2 \pi \mathbb{Z}$. In particular, for $x \in \Tone$, we consider the associated coset $[x] \in \mathbb{R} / 2 \pi \mathbb{Z}$, $[x] := \{ x + 2 \pi k: k \in \mathbb{Z} \}$. Note that the element $x \in \Tone$ is the unique member of the coset $[x]$ contained in the interval $[-\pi, \pi]$; except in the notable case $x = \{\pi, - \pi \}$, where the points $\pi$ and $-\pi$ are both members of the coset $[\pi]$ and they are identified in $\Tone$. Moreover, for $x \in \Tone$ we see that the inverse element $-[x] \in \mathbb{R} / 2 \pi \mathbb{Z}$ corresponds to $- x \in \Tone$. Then, for $z \in \Tone$, define the translation operator $T_z(y) := y - z,$ where $y - z \in \Tone$ is the unique element in $\Tone$ associated with the coset $[y] - [z] \in \mathbb{R} / 2 \pi \mathbb{Z}$. Note that the metric $\dT$ is invariant under translations on $\Tone$, i.e. $\dT(T_z(x),T_z(y)) = \dT(x,y)$ for any $x,y,z \in \Tone$.

%%%%%%%%%%%%%% LOCALIZED OPERATOR  %%%%%%%%%%%%%%%%%%%%%%%%%%%

\subsection{Localized Coefficient}

We begin by localizing the function $b$ to open sets of the form $\mathbb{B}_{\Tone}(z, \varepsilon)$, for $z \in \Tone$ and $\varepsilon \in (0,\frac12)$. We define cut-off functions and `local retractions' which work together to accomplish this goal. For the cut-off functions, choose $X \in C^1 (\Tone)$ such that 
\begin{equation*}\label{Eqn:CutOffXConditions}
\text{supp} \, X \subset (-1, 1) \quad \text{and} \quad X|_{[- \frac12, \frac12]} \equiv 1.
\end{equation*}
Then, define $X_z := X \circ T_z$ -- the cut-off function centered at $z \in \Tone$ -- and notice that $X_z \in C^1(\Tone)$ with supp$(X_z) \subset \mathbb{B}_{\Tone} (z,1)$ for every $z \in \Tone$.

For our `local retractions' we define $r_{\varepsilon} : [-1,1] \rightarrow [-\varepsilon, \varepsilon]$, for $\varepsilon \in (0,\frac12)$, as
\begin{equation}\label{Eqn:LocalRetractionDefined}
r_{\varepsilon}(x) := \begin{cases} x & \text{if $x \in [-\varepsilon, \varepsilon]$},\\ 
\varepsilon & \text{if $x \in (\varepsilon, 1]$},\\
- \varepsilon & \text{if $x \in [- 1, - \varepsilon)$}.  \end{cases}
\end{equation}
Then, for $z \in \Tone$ arbitrary, we define $r_{z, \varepsilon} := T_{-z} \circ r_{\varepsilon} \circ T_z$, the local retraction centered at $z$, which maps the closed neighborhood $\overline{\mathbb{B}}_{\Tone}(z, 1)$ to $\overline{\mathbb{B}}_{\Tone}(z, \varepsilon)$. 

\begin{prop}\label{Prop:RisLipschitz}
For $\varepsilon \in (0, \frac12)$, $r_{\varepsilon}$ is Lipschitz continuous from $[-1,1]$ to $[- \varepsilon,  \varepsilon]$, with Lipschitz constant 1. Consequently, $r_{z, \varepsilon}$ is Lipschitz continuous (with respect to the metric $\dT$) from $\overline{\mathbb{B}}_{\Tone}(z,1)$ to $\overline{\mathbb{B}}_{\Tone}(z,  \varepsilon)$ for all $z \in \Tone$, $\varepsilon \in (0, \frac12)$.
\end{prop}
\begin{proof}
By considering cases for points $x, y \in \Tone$, the first claim is easily verified. Furthermore, notice that $\dT(x,y) = |x - y|$ for $x, y \in [-1,1]$, so that $r_{\varepsilon}$ is Lipschitz in the metric $\dT$ on $[-1,1]$. Then the second claim follows from invariance of the metric $\dT$ under translations $T_z$, for $z \in \Tone$.
\end{proof}

Now, given a function $b \in \mathbb{C}^{\Tone}$, we combine these `local retractions' and `cut-off' functions to define the functions
\[
b_{z,\varepsilon}(x) := \begin{cases} X_z(x) \left[ b \circ r_{z,\varepsilon}(x) - b(z) \right] & \text{ if $x \in \mathbb{B}_{\Tone}(z,1)$,}\\
0 & \text{ otherwise, }  \end{cases} \qquad z \in \Tone, \, \varepsilon \in (0,{\textstyle \frac12}),
\]
which essentially compare the local behavior of $b$ against a fixed value $b(z)$.
Before we make use of these `localized coefficients', we establish the following results regarding their regularity.

\begin{lem}\label{Lem:PerturbedBhAlpha} Let $b \in h^{\alpha}(\Tone)$ for $\alpha \in (0,1)$. Then the following results hold:\hspace{1em}
\begin{enumerate}
	\vspace{.3 em}
	\item $b_{z, \varepsilon} \in h^{\alpha}(\Tone)$ for $\varepsilon \in (0,\frac12), \; z \in \Tone$,
	\vspace{.5 em}
	\item $\displaystyle \lim_{\varepsilon \rightarrow 0^+} \sup_{z \in \Tone} \| b_{z, \varepsilon} \|_{h^{\alpha}} = 0$.
\end{enumerate}
\end{lem}

\begin{proof}

First notice, since $b \in h^{\alpha}(\Tone)$, it follows from the intrinsic characterization of little-H$\ddo$lder
spaces \eqref{IntrisicLittleHolder} that
%\[
%\lim_{\delta \rightarrow 0} \sup_{\substack{x, y \in \Tone\\ 0 < \dT(x,y) < \delta}} \frac{|b(x) - b(y)|}{\dT^{\alpha}(x,y)} = 0.
%\]
%Consequently, 
for $\varepsilon \in (0,\frac12)$, there exists $C(\varepsilon) > 0$ such that 
\begin{equation}\label{GoesToZero}
\sup_{z \in \Tone} [ b \, ]_{\alpha, \overline{\mathbb{B}}(z, \varepsilon)} = C(\varepsilon) \longrightarrow 0 \quad \text{as} \; \varepsilon \rightarrow 0^+.
\end{equation}

\medskip
\noindent
Now, let $z \in \Tone$ be a fixed sample point and $\varepsilon \in (0,\frac12)$. To see that $b_{z,\varepsilon}$ has the necessary regularity, we make use of Proposition~\ref{Prop:IntrinsicDef}(c) and Proposition~\ref{Prop:RisLipschitz}. In particular, let $x, y \in \Tone$ and consider the following cases: 
\begin{itemize} \vspace{.5em}
	
	\item {$x,y \in \mathbb{B}_{\Tone}(z, 1)$:} Then $r_{z,\varepsilon}(x), r_{z,\varepsilon}(y) \in \overline{\mathbb{B}}_{\Tone}(z, \varepsilon)$ and
		\begin{align}
			\nonumber | b_{z,\varepsilon}(x&) - b_{z,\varepsilon}(y) | = \left| X_z(x) \big( b(r_{z,\varepsilon}(x)) - b(z) \big) - X_z(y) \big( b(r_{z,\varepsilon}(y)) - b(z) \big)  \right|\\
			\nonumber &\leq | X_z(x)||b(r_{z,\varepsilon}(x)) - b(r_{z,\varepsilon}(y))| + | X_z(x) - X_z(y) || b(r_{z,\varepsilon}(y)) - b(z) | \\
			\nonumber &\leq \Big( \| X_z \|_{C(\Tone)} \dT^{\alpha}(r_{z,\varepsilon}(x),r_{z,\varepsilon}(y)) + \| X_z' \|_{C(\Tone)} \, \dT(x,y) \, \dT^{\alpha}(r_{z,\varepsilon}(y),z) \Big) [ b \, ]_{\alpha, \overline{\mathbb{B}}(z, \varepsilon)}\label{Eqn:Case1}\\
			&\leq \Big( \| X_z \|_{C(\Tone)} + \| X_z' \|_{C(\Tone)} \; \dT^{1 - \alpha}(x,y) \; \varepsilon^{\alpha} \Big) [ b \, ]_{\alpha, \overline{\mathbb{B}}(z, \varepsilon)} \dT^{\alpha}(x,y)
		\end{align} \vspace{.5em}

	\item {$x \in \mathbb{B}_{\Tone}(z, 1), \, y \in \Tone \setminus \mathbb{B}_{\Tone}(z,1)$:} Then $X_z(y) = 0$ and
		\begin{align}
			\nonumber | b_{z,\varepsilon}(x) &- b_{z,\varepsilon}(y) | = | b_{z,\varepsilon}(x) | = | X_z(x) \big( b( r_{z,\varepsilon}(x)) - b(z) \big)|\\
			\nonumber &\leq |X_z(x) - X_z(y)| \; \dT^{\alpha}(r_{z,\varepsilon}(x),z) \; [ b \, ]_{\alpha, \overline{\mathbb{B}}(z, \varepsilon)}\\
			\nonumber &\leq  \| X_z' \|_{C(\Tone)} \; \dT(x,y) \; \varepsilon^{\alpha} \; [ b \, ]_{\alpha, \overline{\mathbb{B}}(z, \varepsilon)}\\
			&\leq \Big( \| X_z' \|_{C(\Tone)} \; \dT^{1 - \alpha}(x,y) \; \varepsilon^{\alpha}  \Big) \; [ b \, ]_{\alpha, \overline{\mathbb{B}}(z, \varepsilon)} \; \dT^{\alpha}(x,y)\label{Eqn:Case2}
		\end{align} \vspace{.5em}
		
\end{itemize}
Together with the trivial case $x, y \in \Tone \setminus \mathbb{B}_{\Tone} (z, 1)$ -- where $X_z(x) = X_z(y) = 0$ -- this is enough to see that $b_{z, \varepsilon} \in C^{\alpha}(\Tone)$ with the $\alpha$-H$\ddo$lder norm of $b_{z, \varepsilon}$ bounded as
\begin{equation*}\label{Eqn:bejAlphaNormBound}
[ b_{z, \varepsilon} ]_{\alpha, \Tone} \leq \big( \| X_z \|_{C(\Tone)} + \varepsilon^{\alpha} \; \pi^{1 - \alpha} \; \| X_z' \|_{C(\Tone)} \big) \; [ b \, ]_{\alpha, \overline{\mathbb{B}}(z, \varepsilon)}.
\end{equation*}
Furthermore, we can see that $\displaystyle \| b_{z,\varepsilon} \|_{C(\Tone)} \leq  \varepsilon^{\alpha} \; \| X_z \|_{C(\Tone)} \; [ b \, ]_{\alpha, \overline{\mathbb{B}}(z, \varepsilon)}$ so that the $C^{\alpha}$-norm of $b_{z, \varepsilon}$ is bounded as
\begin{equation}\label{Eqn:hAlphaNormBound}
\| b_{z,\varepsilon} \|_{C^{\alpha}} \leq \big( \left( 1 + \varepsilon^{\alpha} \right) \| X_z \|_{C(\Tone)} + \varepsilon^{\alpha} \, \pi^{1 - \alpha} \, \| X_z' \|_{C(\Tone)} \big) [ b ]_{\alpha, \overline{\mathbb{B}}(z, \varepsilon)}.
\end{equation}
Hence, by the property \eqref{GoesToZero} and the  
%it follows that $\| b_{z,\varepsilon} \|_{C^{\alpha}} \rightarrow 0$ as $\varepsilon \rightarrow 0^+$.
%Moreover, 
inequalities \eqref{Eqn:Case1} and \eqref{Eqn:Case2}, we see that 
%provide the necessary uniformity to see that
\[
\lim_{\delta \rightarrow 0^+} \sup_{\substack{ x,y \in \Tone\\ 0 < \dT(x,y) < \delta}} \frac{|b_{z,\varepsilon}(x) - b_{z,\varepsilon}(y)|}{\dT^{\alpha}(x,y)} = 0,
\]
which demonstrates $b_{z,\varepsilon} \in h^{\alpha}(\Tone)$ as claimed in (a).
Now the second claim follows from \eqref{GoesToZero} and \eqref{Eqn:hAlphaNormBound}.

\end{proof}

%Now that we have seen how to reliably localize our coefficient function, we can isolate small perturbations of the operator $\mathcal{A}$ and prove our general result. First, however, we introduce some notation and conventions for sampling data from $\Tone$.

\subsection{Partition and Generation Result}

For $\varepsilon \in (0, \frac12)$, let $n(\varepsilon) := \left\lceil  \frac{2 \pi}{\varepsilon}\right\rceil$, where $\left\lceil a \right\rceil$ denotes the smallest integer $n$ such that $n \geq a$, $a \in \mathbb{R}$. Now, let $\{ x_{\varepsilon, j} : j = 1, \ldots, n(\varepsilon) \} \subset \Tone$ be a collection of sample points from $\Tone$ so that $x_{\varepsilon, 1} = - \pi$ and $x_{\varepsilon, j} = x_{\varepsilon, (j-1)} + \varepsilon, \, j = 2, \ldots, n(\varepsilon) $. Further, define $\Omega_{\varepsilon} := \{ \mathbb{B}_{\Tone}(x_{\varepsilon, j}, \varepsilon): j = 1, \ldots, n(\varepsilon) \},$ which is a finite open cover for $\Tone$, and let $\Pi_{\varepsilon} := \{ \pi^2_{\varepsilon, j} \} \subset C^{\infty}(\Tone)$ be a resolution of unity subordinate to $\Omega_{\varepsilon}.$ In particular, $\Pi_{\varepsilon}$ is a collection of infinitely differentiable functions such that
\[
\text{supp} (\pi_{\varepsilon, j}) \subset \mathbb{B}_{\Tone}(x_{\varepsilon, j}, \varepsilon), \; j = 1, \ldots, n(\varepsilon), 
\qquad \text{and} \qquad
\sum_{j = 1}^{n(\varepsilon)} \pi_{\varepsilon, j}^2 (x) = 1, \; x \in \Tone.
\]

\noindent
Now we are prepared to prove the following result, which is a generalization of Theorem~\ref{Thm:GenerationWithConstant} to the case of non-constant coefficients. The method of the proof is motivated by results in \cite{AM01,AMHS94}.

%% file: GenerationofAnalyticSemigroup-J3E.tex
%\subsection{Generation of $\mathcal{A}_p$}

\begin{lem}\label{MainResult}
Let $m \in \mathbb{N}$, $\alpha \in \mathbb{R}_+ \setminus \mathbb{Z}$ and consider the differential operator $\mathcal{A}_p := \mathcal{A}_p(\cdot,D) := b(\cdot) \,  D^{2m}$ with coefficient $b \in h^{\alpha}(\Tone)$. If $\mathcal{A}_p$ is uniformly elliptic, %with constant $c_1 > 0$,
%and $c_2 \geq c_1 > 0$ is chosen so that $\| b \|_{C(\Tone)} \leq c_2$, 
then $- \mathcal{A}_p$ generates a (strongly continuous) analytic semigroup on $h^{\alpha}(\Tone)$ with domain $h^{2m + \alpha}(\Tone)$.
i.e. $\mathcal{A}_p \in \mathcal{H}(h^{2m + \alpha}(\Tone),h^{\alpha}(\Tone))$.
%Let $\alpha \in \mathbb{R}_+ \setminus \mathbb{Z}, \; 0 < c_1 \leq c_2, \; \text{and } b \in h^{\alpha}(\Tone)$ so that $0 < c_1 \leq b(x) \leq c_2$ for $x \in \Tone$. Then $\displaystyle \mathcal{A}_p := b(\cdot) D^{2m} \in \mathcal{H} ( h^{2m + \alpha}(\Tone), h^{\alpha}(\Tone)).$
\end{lem}

\begin{proof} 
Fix $\omega > 0$ and $b \in h^{\alpha}(\Tone).$ By assumption, there exist constants $c_1$ and $c_2$, with $c_2 \geq c_1 > 0$ such that $b(\Tone) \subset \{ z \in \mathbb{C}: \re \, z \geq c_1 \} \cap \{ z \in \mathbb{C}: |z| \leq c_2 \}$.

(i) First we demonstrate that it suffices to prove the result for $\alpha \in (0,1)$. Suppose that the claim holds for $\alpha \in (0,1)$ and let $\beta := \alpha + 1$. In particular, we assume $b \in h^{\beta}(\Tone)$ and $\mathcal{A}_p = b(\cdot)D^{2m} \in \mathcal{H}(h^{2m + \alpha}(\Tone),h^{\alpha}(\Tone))$. It follows that $(\lambda + \mathcal{A}_p): h^{2m + \alpha}(\Tone) \rightarrow h^{\alpha}(\Tone)$ is invertible for $\re \, \lambda \geq \omega$ and we have the resolvent estimates 
\begin{equation}\label{Eqn:fResolvent}
\|(\lambda + \mathcal{A}_p)^{-1} \|_{\mathcal{L}(h^{\alpha},h^{2m + \alpha})} \leq \kappa,  \qquad
|\lambda| \|(\lambda + \mathcal{A}_p)^{-1} \|_{\mathcal{L}(h^{\alpha})} \leq \kappa,
\end{equation}
for $\re \, \lambda \geq \omega$, for some $\omega > 0$ and some $\kappa \geq 1$.
Now, fix $\lambda \in \mathbb{C}$ with $\re \, \lambda \geq \omega$ and consider $f \in h^{\beta}(\Tone)$. Then $f \in h^{\alpha}(\Tone),$ by Proposition~\ref{Thm:hAlphaCAlphaClosure}(a), and we define $u := (\lambda + \mathcal{A}_p)^{-1} f \in h^{2m + \alpha}(\Tone)$. Then $u$ satisfies the equation $(\lambda + \mathcal{A}_p)u = f$ and, differentiating this equation, we see that 
\begin{equation*}
(\lambda + \mathcal{A}_p)u' = f' - b' \, u^{(2m)},
\end{equation*}
where, a priori, we know that $u' \in h^{2m - 1 + \alpha}(\Tone)$. However, notice that\\ $f', b', u^{(2m)} \in h^{\alpha}(\Tone)$ and $b' \, u^{(2m)} \in h^{\alpha}(\Tone)$, since $h^{\alpha}(\Tone)$ is a Banach algebra, so that $u' = (\lambda + \mathcal{A}_p)^{-1} \big( f' - b' \, u^{(2m)} \big) \in h^{2m + \alpha}(\Tone)$. Hence, we see that $u \in h^{2m + \beta}(\Tone)$ and 
\begin{align*}
\| u \|_{h^{2m + \beta}(\Tone)} &= \sum_{k=0}^{2m + 1} \| u^{(k)} \|_{C(\Tone)} + [ u^{(2m + 1)} \,]_{\alpha, \Tone} = \| u' \|_{h^{2m + \alpha}(\Tone)} + \| u \|_{C(\Tone)}\\
&\leq \kappa \big( \|f' - b' \, u^{(2m)}\|_{h^{\alpha}(\Tone)} + \| u \|_{C(\Tone)} \big)\\
&\leq \kappa \big( \| f' \|_{h^{\alpha}(\Tone)} + (1 \vee \| b' \|_{h^{\alpha}(\Tone)}) \| u \|_{h^{2m + \alpha}(\Tone)} \big)\\
&\leq \kappa \big( \| f' \|_{h^{\alpha}(\Tone)} + \kappa (1 \vee \| b' \|_{h^{\alpha}(\Tone)}) \| f \|_{h^{\alpha}(\Tone)} \big)\\
&\leq K(\kappa, b) \big( \| f' \|_{C(\Tone)} + [ f' \, ]_{\alpha, \Tone} + \| f \|_{C(\Tone)} + [ f \, ]_{\alpha, \Tone} \big)\\
&\leq K(\kappa, b) \big( (1 + \pi^{1 - \alpha}) \| f' \|_{C(\Tone)} + \| f \|_{C(\Tone)} + [ f' \, ]_{\alpha, \Tone} \big) \leq \tilde{K} \| f \|_{h^{\beta}(\Tone)}.
\end{align*}
So that $\|(\lambda + \mathcal{A}_p)^{-1} \|_{\mathcal{L}(h^{\beta},h^{2m + \beta})} \leq \tilde{K}$ for $\re \, \lambda \geq \omega$. Meanwhile, in a similar fashion, we see that
\begin{align*}
|\lambda| \| u \|_{h^{\beta}(\Tone)} = |\lambda| \big( \| u \|_{C(\Tone)} + \| u' \|_{C(\Tone)} + [ u' \, ]_{\alpha, \Tone} \big) \leq \tilde{K} \| f \|_{h^{\beta}(\Tone)},
\end{align*}
holds for $\re \, \lambda \geq \omega$. Hence, it follows that $| \lambda | \|(\lambda + \mathcal{A}_p)^{-1} \|_{\mathcal{L}(h^{\beta})} \leq \tilde{K}$ for $\re \, \lambda \geq \omega$ and so the claim holds for $\beta = \alpha + 1$. Then, we extend the result to any $\beta > 1$, $\beta \notin \mathbb{Z}$, by induction on $\alpha$.

(ii) Now, we demonstrate the claim for $\alpha \in (0,1)$. By uniform ellipcticity of $\mathcal{A}_p$, it follows from Theorem~\ref{Thm:GenerationWithConstant} that there exists some constant $\kappa = \kappa(\omega, c_1, c_2) \geq 1$ such that $\mathcal{A}_p(x_0) := b(x_0) D^{2m} \in \mathcal{H}(h^{2m + \alpha}(\Tone), h^{\alpha}(\Tone), \kappa, \omega)$ for any fixed $x_0 \in \Tone.$ Now, fix $\eta$ so that $0 < \eta < 1 / \kappa$. By Lemma~\ref{Lem:PerturbedBhAlpha}(b), there exists $\varepsilon_0 > 0$ with associated sampling set $\{x_j\} := \{ x_{\varepsilon_0,j} \}$ and partition $\Omega := \Omega_{\varepsilon_0} = \{ \mathbb{B}_{\Tone}(x_j, \varepsilon_0) \}$, $j = 1, \ldots, n := n(\varepsilon_0)$, such that 
\begin{equation}\label{Eqn:PerturbedBBound}
\sup_{j = 1, \ldots, n} \| b_{j} \|_{h^{\alpha}} < \eta, \qquad \text{where} \quad b_{j} := b_{x_j,\varepsilon_0} \,.
\end{equation}
Moreover, by Lemma~\ref{Lem:PerturbedBhAlpha}(a) and the fact that $h^{\alpha}(\Tone)$ is a Banach Algebra, the operator $b_{j}(\cdot) D^{2m}$ is in $\mathcal{L}(h^{2m + \alpha}(\Tone), h^{\alpha}(\Tone))$ with $\| b_{j}(\cdot) D^{2m} \|_{\mathcal{L}(h^{2m + \alpha}, h^{\alpha})} \leq \| b_{j} \|_{h^{\alpha}}$, for $j = 1, \ldots, n$. Hence, by \cite[Theorem 1.3.1(i)]{AM95} and \eqref{Eqn:PerturbedBBound} we can see that perturbations of $\mathcal{A}_p(x_j)$ remain in the class $\mathcal{H}$, namely
\begin{equation*}\label{Eqn:PerturbedOperatorType}
\mathcal{A}_j := \left[ b(x_j) + b_{j}(\cdot) \right] D^{2m} \in \mathcal{H}\left( h^{2m + \alpha}(\Tone), h^{\alpha}(\Tone),\frac{\kappa}{1 - \kappa \eta}, \omega \right), \qquad \! j = 1, \ldots, n.
\end{equation*}
In particular, this implies that $\{ \lambda \in \mathbb{C}: \re \, \lambda \geq \omega \} \subset \rho(- \mathcal{A}_j )$ and the resolvent estimates 
\begin{equation}\label{Eqn:AjResolventEstimates}
\begin{split}
|\lambda| \| (\lambda + \mathcal{A}_j)^{-1} \|_{\mathcal{L}(h^{\alpha})} \leq \frac{\kappa}{1 - \kappa \eta},\\
\| (\lambda + \mathcal{A}_j)^{-1} \|_{\mathcal{L}(h^{\alpha},h^{2m + \alpha})} \leq \frac{\kappa}{1 - \kappa \eta},
\end{split}
\end{equation} 
hold uniformly for $\re \, \lambda \geq \omega$ and $j = 1, \ldots, n$.

Let $\Pi := \Pi_{\varepsilon_0} = \{ \pi^2_j \}$ be a resolution of unity subordinate to $\Omega$, where 
%so that $\text{supp} \, \pi_j \subset \mathbb{B}_{\Tone}(x_j,\varepsilon_0)$ for $j = 1, \ldots, n$ and $\sum_{j = 1}^n \pi_j^2 = 1$, 
we also insist that $\| \pi_j \|_{h^{\alpha}}, \, \| \pi_j \|_{h^{2m + \alpha}} \leq M$ uniformly in $j$, for some $M = M (\varepsilon_0) \geq 1$. Now, define the composite little-H$\ddo$lder spaces
\begin{equation*}\label{Eqn:hAlphanDefined}
\left( h^{\sigma}(\Tone) \right)^n := \left\{ ( f_j )_{j \in \mathbb{N}} \in \ell^{\infty} \left( h^{\sigma}(\Tone) \right) : f_j = 0 \text{ for } j > n \right\}, \qquad \sigma \in \mathbb{R}_+ \setminus \mathbb{Z}. 
\end{equation*}
Then, it is easy to see that $\left( h^{\sigma}(\Tone) \right)^n$ is a Banach space, with the norm topology inherited from $\ell^{\infty} \left( h^{\sigma}(\Tone) \right)$. Moreover, we have the following retraction and coretraction
\begin{align}
R:& \, (h^{\sigma}(\Tone))^n \rightarrow h^{\sigma}(\Tone) \qquad \text{where } R\left( (f_j)_j \right) := \sum_{j = 1}^n \pi_j f_j\\
R^C&: h^{\sigma}(\Tone) \rightarrow (h^{\sigma}(\Tone))^n \qquad \text{where } R^C \left( u \right) := \left( \pi_j u \right)_j.
\end{align}
With finiteness of the partition $\Omega$ and the properties of the resolution of unity $\Pi$, we easily see that $R \in \mathcal{L}\left( (h^{\sigma}(\Tone))^n, h^{\sigma}(\Tone) \right)$ and $R^C \in \mathcal{L} \left( h^{\sigma}(\Tone), (h^{\sigma}(\Tone))^n \right)$ with
$R \circ R^C  = id_{h^{\sigma}(\Tone)}$ and 
\begin{equation}\label{Eqn:RetractNormBounds}
\| R \|_{\mathcal{L}\left( (h^{\sigma}(\Tone))^n, h^{\sigma}(\Tone) \right)} \leq n M, \qquad \| R^C \|_{\mathcal{L} \left( h^{\sigma}(\Tone), (h^{\sigma}(\Tone))^n \right)} \leq M, \qquad \sigma \in \{ \alpha, 2m + \alpha \}.
\end{equation}

We will make use of $R$ and $R^C$ together with the spaces $\left( h^{\sigma} (\Tone) \right)^n$ to construct a left and right inverse for $(\lambda + \mathcal{A}_p)$, for $\re \, \lambda \geq \omega_0 \geq \omega$ sufficiently large. Toward this goal, we define the following operators:
\begin{itemize}
	
	\item $\displaystyle \Lambda : \left( h^{2m + \alpha}(\Tone) \right)^n \rightarrow \hAlphaN$ defined by $\Lambda (f_j)_j := ( \mathcal{A}_j f_j)_j.$ Then $\Lambda \in \mathcal{L} \left( \left( h^{2m + \alpha}(\Tone) \right)^n, \hAlphaN \right)$ with $\displaystyle \| \Lambda \| \leq \sup_{j = 1, \ldots, n} \| \mathcal{A}_j \|_{\mathcal{L} \left(h^{2m + \alpha},h^{\alpha} \right)}$.

	\item $\displaystyle B_j := \pi_j \mathcal{A}_j - \mathcal{A}_j \pi_j = [\pi_j, \mathcal{A}_p]$ the commutator of $\pi_j$ and $\mathcal{A}_p$, $j = 1, \ldots, n$. The second expression for $B_j$ follows from the fact that\\ supp$(\pi_j) \subset \mathbb{B}_{\Tone}(x_j, \varepsilon_0)$ and $b_{j}(x) = b(x) - b(x_j)$ for $x \in \mathbb{B}_{\Tone}(x_j, \varepsilon_0)$, so that $\mathcal{A}_p$ and $\mathcal{A}_j$ coincide on supp$(\pi_j)$. Moreover, the highest order terms are eliminated in $B_j$ so that we have $B_j \in \mathcal{L} \left( h^{(2m-1) + \alpha}(\Tone), h^{\alpha}(\Tone) \right)$ with $\| B_j \| \leq C(m) \| \pi_j \|_{h^{2m + \alpha}} \leq C(m) M$.
	
	\item $\displaystyle \mathcal{B} : h^{(2m-1) + \alpha}(\Tone) \rightarrow \hAlphaN$ defined by $\mathcal{B}f := \left( B_j f \right)_j$. Then\\ $\mathcal{B} \in \mathcal{L} \left( h^{(2m-1) + \alpha}(\Tone), \hAlphaN \right)$ with\\ $\displaystyle \| \mathcal{B} \| \leq \sup_{j = 1, \ldots, n} \| B_j \|_{\mathcal{L} \left( h^{(2m-1) + \alpha}, h^{\alpha} \right)} \leq C(m) M$.

	\item $\displaystyle \mathcal{D} : (h^{(2m-1) + \alpha}(\Tone))^n \rightarrow h^{\alpha}(\Tone)$ defined by $\displaystyle \mathcal{D}(f_j)_j := \sum_{j=1}^n B_j f_j$. Then\\ 
	$\mathcal{D} \in \mathcal{L}((h^{(2m-1) + \alpha}(\Tone))^n, h^{\alpha}(\Tone))$  with $\| \mathcal{D} \|%_{\mathcal{L}((h^{(2m-1) + \alpha})^n,h^{\alpha})} 
	\leq n\,C(m)M$.
	
	\item $\displaystyle C_{j,k}(\lambda):= B_j \circ \pi_k \circ \left( \lambda + \mathcal{A}_k \right)^{-1}$, $j,k = 1, \ldots, n$, $\re \, \lambda \geq \omega$. We easily see that $C_{j,k} \in \mathcal{L}(h^{\alpha}(\Tone))$. Moreover, since $B_j$ maps $h^{(2m-1) + \alpha}(\Tone)$ to $h^{\alpha}(\Tone)$ we can consider the mapping $(\lambda + \mathcal{A}_k)^{-1}$ from $h^{\alpha}(\Tone)$ to $h^{(2m-1) + \alpha}(\Tone)$. In this way, we take advantage of the interpolation result for little-H$\ddo$lder spaces, Proposition~\ref{Thm:hAlphaCAlphaClosure}(b), in conjunction with the resolvent estimates \eqref{Eqn:AjResolventEstimates} on $\mathcal{A}_k$, to see that 
\begin{align}\label{Eqn:BoundIntoh3Alpha}
\nonumber \| (\lambda + \mathcal{A}_k)^{-1} \|&_{\mathcal{L}(h^{\alpha},h^{(2m-1) + \alpha})} \leq \| (\lambda + \mathcal{A}_k)^{-1} \|_{\mathcal{L}(h^{\alpha},h^{2m + \alpha})}^{1 - 1/2m} \| (\lambda + \mathcal{A}_k)^{-1} \|_{\mathcal{L}(h^{\alpha})}^{1/2m}\\ 
&\nonumber \leq \left( \frac{\kappa}{1 - \kappa \eta} \right)^{1 - 1/2m} \left( \frac{\kappa}{1 - \kappa \eta} \right)^{1/2m} | \lambda |^{-1/2m}\\
&\leq \left( \frac{\kappa}{1 - \kappa \eta} \right) | \lambda |^{-1/2m} = \tilde{c} | \lambda |^{-1/2m}.
\end{align} 
Here, we take advantage of the fact that the continuous interpolation method used in Propostion~\ref{Thm:hAlphaCAlphaClosure}(b) is exact. Hence, the $C_{j,k}$ operator norms are bounded as
\begin{equation}\label{Eqn:CjkBound}
\| C_{j,k} \|_{\mathcal{L}(h^{\alpha}(\Tone))} \leq \tilde{c} \, C(m) M^2 | \lambda |^{-1/2m} \qquad j,k = 1, \ldots, n, \; \re \, \lambda \geq \omega.
\end{equation}

	\item $\displaystyle \mathcal{C}(\lambda): \hAlphaN \rightarrow \hAlphaN$ defined $\displaystyle \mathcal{C}(f_j)_j := \left( B_j \sum_{k = 1}^n \pi_k (\lambda + \mathcal{A}_k)^{-1} f_k \right)_{j}$, for $\re \, \lambda \geq \omega$. Notice that $\text{supp} \, \pi_k \subset \mathbb{B}_{\Tone}(x_k, \varepsilon_0)$ and $\text{supp} \, B_j \subset \mathbb{B}_{\Tone}(x_j, \varepsilon_0)$ for $j,k = 1, \ldots, n$, so $C_{j,k}(\lambda) = 0$ for $1 < |j - k| < n-1$. Hence, by \eqref{Eqn:CjkBound}, we can choose $\omega_1 > 0$ large enough to ensure that $\| \mathcal{C}(\lambda) \|_{\mathcal{L}(\hAlphaN)} \leq 1/2$ for $\re \, \lambda \geq \omega_1$.
\end{itemize}

\noindent {\bf Claim 1:} \emph{ For $\re \, \lambda \geq \omega_1$, $\left( \lambda + \Lambda + \mathcal{B}R \right): \left( h^{2m + \alpha}(\Tone) \right)^n \rightarrow \left( h^{\alpha}(\Tone) \right)^n$ is invertible and $L(\lambda) := R \left( \lambda + \Lambda + \mathcal{B}R \right)^{-1} R^C$ is a left inverse for $(\lambda + \mathcal{A}_p)$.}

From the definition and discussion of $\mathcal{C}(\lambda)$ above, we can choose $\omega_1 \geq \omega$ large enough so that $\| \mathcal{C}(\lambda) \|_{\mathcal{L}(\hAlphaN)} \leq 1/2$ for $\re \, \lambda \geq \omega_1$. Hence, by the Neumann series, we see that $\left( id_{\hAlphaN} + \mathcal{C}(\lambda) \right)$ is invertible on $\left( h^{\alpha}(\Tone) \right)^n$ for $\re \, \lambda \geq \omega_1$ and $\| ( id_{\hAlphaN} + \mathcal{C}(\lambda))^{-1} \|_{\mathcal{L}(\hAlphaN)} \leq 2$. Now, for any $(f_j)_j \in \left( h^{2m + \alpha}(\Tone) \right)^n$, $\re \, \lambda \geq \omega_1$, we have
\begin{align*}
	\mathcal{B}R (f_j)_j &= \mathcal{B} \Big( \sum_{k = 1}^n \pi_k f_k \Big) = \Big( B_j \sum_{k = 1}^n \pi_k f_k \Big)_j\\
	 &= \bigg( B_j \sum_{k = 1}^n \pi_k (\lambda + \mathcal{A}_k)^{-1} (\lambda + \mathcal{A}_k) f_k \bigg)_j = \mathcal{C}(\lambda) \Big( (\lambda + \mathcal{A}_j) f_j \Big)_j \\
	\Longrightarrow (\lambda + \Lambda +& \mathcal{B}R)(f_j)_j = (\lambda + \Lambda) (f_j)_j + \mathcal{C}(\lambda) \Big( (\lambda + \mathcal{A}_j) f_j \Big)_j\\
	 & \hspace{4em}= (id_{(h^{\alpha}(\Tone))^n} + \mathcal{C}(\lambda))(\lambda + \Lambda) (f_j)_j .
\end{align*}
Hence, invertibility of $(\lambda + \Lambda + \mathcal{B}R)$ follows from invertibility of $(id_{(h^{\alpha}(\Tone))^n} + \mathcal{C}(\lambda))$ and invertibility of $(\lambda + \Lambda)$, both of which hold if $\re \, \lambda \geq \omega_1 \geq \omega$.
Furthermore, we see that $(\lambda + \Lambda + \mathcal{B}R)^{-1} = (\lambda + \Lambda)^{-1}(id_{(h^{\alpha}(\Tone))^n} + \mathcal{C}(\lambda))^{-1}$, $\re \, \lambda \geq \omega_1$. 

Now, we apply \eqref{Eqn:AjResolventEstimates} to see that 
\[
\| ( \lambda + \Lambda )^{-1} \|_{\mathcal{L}( \hAlphaN )} \leq \left( \frac{\kappa}{1 - \kappa \eta} \right) | \lambda |^{-1},
\]
and so, with \eqref{Eqn:RetractNormBounds}, we get the bound
\begin{equation}\label{Eqn:FullResolventEstimate}
\| L(\lambda) \|_{\mathcal{L}(h^{\alpha})} = \| R(\lambda + \Lambda + \mathcal{B}R)^{-1}R^C \| \leq \left( \frac{2 \kappa }{1 - \kappa \eta} \right)n M^2 | \lambda |^{-1}.
\end{equation}

Finally, to see that $L(\lambda)$ is indeed a left inverse for $(\lambda + \mathcal{A}_p)$, when $\re \, \lambda \geq \omega_1$. Let $u \in h^{2m + \alpha}(\Tone)$ and $\pi_j \in \Pi$, then we see that
\[
\pi_j(\lambda + \mathcal{A}_p)u = (\lambda + \mathcal{A}_j) \pi_j u + B_j u = (\lambda + \mathcal{A}_j) \pi_j u + B_j R R^C u.
\]
Hence, it follows, by exhibiting all components for $j = 1, \ldots, n$, that
\[
R^C (\lambda + \mathcal{A}_p) u = (\lambda + \Lambda + \mathcal{B}R) R^C u \quad \text{and so} \quad L(\lambda)(\lambda + \mathcal{A}_p) = id_{h^{2m + \alpha}(\Tone)}.
\]

\noindent{\bf Claim 2:} \emph{$\left( \lambda + \Lambda - R^C \mathcal{D} \right): \left( h^{2m + \alpha}(\Tone) \right)^n \rightarrow \left( h^{\alpha}(\Tone) \right)^n$ is invertible for $\re \, \lambda \geq \omega_2 \geq \omega$ with $\omega_2$ sufficiently large. Moreover, $R(\lambda) := R \left( \lambda + \Lambda - R^C \mathcal{D} \right)^{-1} R^C$ is a right inverse for $(\lambda + \mathcal{A}_p)$ and $L(\lambda) = R(\lambda) = (\lambda + \mathcal{A}_p)^{-1}$.}

We make use of the same observations that led to the invertibility of $(\lambda + \Lambda + \mathcal{B} R)$ in the previous claim. Notice that
\[
(\lambda + \Lambda - R^C \mathcal{D}) = (id_{\hAlphaN} - R^C \mathcal{D} (\lambda + \Lambda)^{-1})(\lambda + \Lambda), \quad \re \, \lambda \geq \omega,
\]
so it suffices to show that $(id_{\hAlphaN} + R^C \mathcal{D} (\lambda + \Lambda)^{-1})$ is invertible in $\mathcal{L}(\hAlphaN)$, for $\re \, \lambda$ sufficiently large. However, this follows by the Neumann series, taking into account the fact that $\mathcal{D} \in \mathcal{L}((h^{(2m-1) + \alpha}(\Tone))^n, h^{\alpha}(\Tone))$, so that \eqref{Eqn:BoundIntoh3Alpha} implies
\[
\| R^C \mathcal{D} (\lambda + \Lambda)^{-1} \|_{\mathcal{L}(\hAlphaN)} \leq \tilde{c} \, n C(m) M^2 | \lambda |^{-1/2m}.
\]
Hence, we can choose $\omega_2 \geq \omega$ large enough that $\| R^C \mathcal{D} (\lambda + \Lambda)^{-1} \|_{\mathcal{L}(\hAlphaN)} \leq 1/2$ for $\re \, \lambda \geq \omega_2$, which implies invertibility of $(\lambda + \Lambda - R^C \mathcal{D})$. Furthermore, to see that $R(\lambda)$ is a right inverse for $(\lambda + \mathcal{A}_p)$, 
% first notice that $\mathcal{A} \pi_j f = \mathcal{A}_j \pi_j f$ for $j = 1, \ldots, n$, since $\pi_j$ and $b_j$ are both supported within $\mathbb{B}_{\varepsilon_0}(x_j)$. Now, 
let $(f_j)_j \in (h^{2m + \alpha}(\Tone))^n$ and notice that
\begin{align*}
(\lambda + \mathcal{A}_p)R(f_j)_j &= %(\lambda + \mathcal{A}_p) \sum_{j = 1}^n \pi_j f_j = 
\sum_{j = 1}^n (\lambda + \mathcal{A}_p) \pi_j f_j = \sum_{j = 1}^n \left( \pi_j (\lambda + \mathcal{A}_j) f_j - B_j f_j \right)\\
&= \sum_{j = 1}^n \pi_j (\lambda + \mathcal{A}_j) f_j - \sum_{j = 1}^n B_j f_j = R (\lambda + \Lambda) (f_j)_j - \mathcal{D}(f_j)_j\\ 
&= R (\lambda + \Lambda) (f_j)_j - R R^C \mathcal{D}(f_j)_j = R( \lambda + \Lambda - R^C \mathcal{D}) (f_j)_j \,.
\end{align*}
Hence, $(\lambda + \mathcal{A}_p) R(\lambda) = id_{h^{\alpha}(\Tone)}$ and $R(\lambda)$ is a right inverse for $(\lambda + \mathcal{A}_p)$.

Finally, let $\omega_0 = \omega_1 \vee \omega_2$, so that $L(\lambda)$ and $R(\lambda)$ are both defined for $\re \, \lambda \geq \omega_0$. Then $L(\lambda) f = L(\lambda) \big[ (\lambda + \mathcal{A}_p) R(\lambda) \big] f = \big[ L(\lambda) (\lambda + \mathcal{A}_p) \big] R(\lambda) f = R(\lambda) f,$ for $f \in h^{\alpha}(\Tone)$. Hence, $(\lambda + \mathcal{A}_p)$ is invertible for $\re \, \lambda \geq \omega_0$ and $\mathcal{A}_p \in \mathcal{H}(h^{2m + \alpha}(\Tone), h^{\alpha}(\Tone))$ follows from the resolvent estimate \eqref{Eqn:FullResolventEstimate}.

\end{proof}

With this generation result for the principal operator $\mathcal{A}_p$ established, we return to the full elliptic operator $\mathcal{A}$. Making use of perturbation results for generators of analytic semigroups, we prove that -$\mathcal{A}$ generates an analytic semigroup in the scale of little-H$\ddo$lder spaces. 

\begin{thm}\label{FullResult}
Let $m \in \mathbb{N}, \; \alpha \in \mathbb{R}_+ \setminus \mathbb{Z}, b_k \in h^{\alpha}(\Tone), \; \text{for} \; k = 0, \ldots, 2m$, and suppose the operator $\mathcal{A} := \mathcal{A}(\cdot,D) := \sum_{k = 0}^{2m} b_k(\cdot) D^k$ is uniformly elliptic. Then
\[
\mathcal{A} \in \mathcal{H} ( h^{2m + \alpha}(\Tone), h^{\alpha}(\Tone)).
\]

\end{thm}
\begin{proof}
By Lemma~\ref{MainResult} and \cite[Theorem I.1.2.2]{AM95} we can find $\omega > 0$ and $\kappa \geq 1$ such that $\mathcal{A}_p := \mathcal{A}_p(\cdot,D) := b_{\, 2m}(\cdot)D^{2m} \in \mathcal{H}(h^{2m + \alpha}(\Tone),h^{\alpha}(\Tone), \kappa, \omega)$. Now, fix $\eta$ so that $0 < \eta < 1/ \kappa$ and consider the operator $B_1 := B_1(\cdot,D) := b_{\, 2m - 1}(\cdot)D^{2m - 1}.$ For any $f \in h^{2m + a}(\Tone)$, we make use of the interpolation inequality, c.f. \cite[Proposition I.2.2.1]{AM95}, and Young's inequality to see that
\begin{align*}
\| B_1 \, f \|_{h^{\alpha}(\Tone)} &\leq \| b_{\, 2m - 1} \|_{h^{\alpha}(\Tone)} \; \| f^{2m - 1} \|_{h^{\alpha}(\Tone)} \leq \| b_{\, 2m - 1} \|_{h^{\alpha}(\Tone)} \; \| f \|_{h^{2m - 1 + \alpha}(\Tone)}\\
&\leq c \| b_{\, 2m - 1} \|_{h^{\alpha}(\Tone)} \; \big( \| f \|_{h^{\alpha}(\Tone)}^{\frac{1}{2m}} \; \| f \|_{h^{2m + \alpha}(\Tone)}^{\frac{2m - 1}{2m}} \big)\\
&\leq c \| b_{\, 2m - 1} \|_{h^{\alpha}(\Tone)} \; \big( \varepsilon \, \| f \|_{h^{2m + \alpha}(\Tone)} + \tilde{c} \, \varepsilon^{1 - 2m} \; \| f \|_{h^{\alpha}(\Tone)} \big),
\end{align*}
where the last inequality holds for arbitrary $\varepsilon > 0$. Now, if we choose $\varepsilon > 0$ such that $\tilde{\varepsilon} := c \varepsilon \| b_{\, 2m - 1} \|_{h^{\alpha}(\Tone)} < \eta$, it follows from \cite[Theorem 1.3.1(ii)]{AM95} that 
\[
\mathcal{A}_p + B_1 \in \mathcal{H} \left( h^{2m + \alpha}(\Tone), \; h^{\alpha}(\Tone), \; \frac{\kappa}{1 - \kappa \tilde{\varepsilon}}, \; \omega \; \vee \; \frac{ c \tilde{c} \, \varepsilon^{1 - 2m} \; \| b_{\, 2m - 1} \|_{h^{\alpha}(\Tone)}}{\tilde{\varepsilon}} \right).
\]
Now the theorem follows by repeating this argument for the remaining lower-order terms of the operator $\mathcal{A}$.
\end{proof}

\begin{rem}
Notice that the results of Theorem~\ref{FullResult} also hold in the setting of classic H$\ddo$lder spaces $C^{\sigma}(\Tone),$
though one must still take coefficients from the little-H$\ddo$lder spaces to preserve smallness of localized coefficients, c.f. Lemma~\ref{Lem:PerturbedBhAlpha}(b).
One notable difference when considering these analogous results in the classic H$\ddo$lder setting is that 
the semigroups generated are no longer strongly continuous, due to a lack of dense embeddings in this setting.
For the methods leading to maximal regularity that follow, strong continuity of semigroups is necessary, 
so the results in the little-H$\ddo$lder setting are required for our purposes.
\end{rem}

%% file: Conclusion-J3E.tex
\section{Maximal Regularity and The Inhomogeneous Problem}

Now we return to the task of finding solutions to the inhomogeneous problem 
\begin{equation}\label{Eqn:CP2}
\begin{cases}
\partial_t u(t,x) + \mathcal{A}(x,D)u(t,x) = f(t,x), & \text{$t > 0, \, x \in \mathbb{R}$}\\
u(0,x) = u_0(x), & \text{$x \in \mathbb{R}$},
\end{cases}
\end{equation} 
Given an interval $J := [0, T]$ %(or $J = \mathbb{R}_+$) 
and $\dot{J} := J \setminus \{ 0 \},$ we say that $u:[t \mapsto u(t,\cdot)]$ is a \emph{classical solution} to \eqref{Eqn:CP2} if
\[
u \in C^1(\dot{J}, C(\Tone)) \cap C(\dot{J},C^{2m}(\Tone)) \cap C(J, C(\Tone)),
\]
and $u$ satisfies \eqref{Eqn:CP2}. Following results of DaPrato, Grisvard and Angenent, we will show how the analytic semigroup generation result of Theorem~\ref{FullResult} leads to existence and uniqueness of solutions to \eqref{Eqn:CP2}, with maximal regularity of solutions. We begin by defining function spaces which define the temporal regularity of solutions (i.e. mapping the interval $J$ into the little-H$\ddo$lder spaces), then we define a class of maximal regularity and use properties of maximal regularity to demonstrate existence and uniqueness of solutions.

%%%%%%%%%%%%%%% Singular Continuity and Max Reg %%%%%%%%%%%%%%%%%%%%%%%%

\subsection{Function Spaces and Maximal Regularity}

Addressing temporal regularity of solutions, let $\mu \in (0,1], \, J := [0,T]$, for some $T > 0$, and let $E$ be a Banach space. Following the notation of \cite{CS01}, we define spaces of functions defined on $\dot{J} := J \setminus \{ 0 \}$ with prescribed singularity at 0. Namely, define
\begin{equation}\label{Eqn:SingularContinuity}
\begin{aligned}
&BUC_{1 - \mu}(J,E) := \bigg\{ u \in C(\dot{J}, E): [t \mapsto t^{1 - \mu} u(t)] \in BUC(\dot{J},E) \; \text{and} \\ 
& \hspace{9em} \lim_{t \rightarrow 0^+} t^{1 - \mu} \| u(t) \|_E = 0 \bigg\}, \quad \mu \in (0,1)\\
&\| u \|_{C_{1 - \mu}} := \sup_{t \in J} t^{1 - \mu} \| u(t) \|_E,
\end{aligned}
\end{equation}
where $BUC$ denotes the space consisting of bounded, uniformly continuous functions. It is easy to verify that $BUC_{1 - \mu}(J,E)$ is a Banach space when equipped with the norm $\| \cdot \|_{C_{1 - \mu}}$. Moreover, we define the subspace
\[
BUC_{1 - \mu}^1(J,E) := \left\{ u \in C^1(\dot{J},E) : u, \dot{u} \in BUC_{1 - \mu}(J, E) \right\}, \quad \mu \in (0,1)
\]
and we set
\[
BUC_0 (J,E) := BUC(J,E) \qquad BUC^1_0(J,E) := BUC^1(J,E).
\]

Now, fix $\mu \in (0,1]$ and consider the spaces 
\[
\begin{split}
&\mathbb{E}_0(J) := BUC_{1 - \mu}(J,h^{\alpha}(\Tone)), \qquad \alpha \in \mathbb{R}_+ \setminus \mathbb{Z}\\
&\mathbb{E}_1(J) := BUC^1_{1 - \mu}(J,h^{\alpha}(\Tone)) \cap BUC_{1-\mu}(J,h^{2m + \alpha}(\Tone)),
\end{split}
\] 
where $\mathbb{E}_1(J)$ is a Banach space with the norm 
\[
\| u \|_{\mathbb{E}_1(J)} := \sup_{t \in \dot{J}} t^{1 - \mu} \Big( \| \dot{u}(t) \|_{h^{\alpha}} + \| u(t) \|_{h^{2m + \alpha}} \Big).
\]
It follows that the trace operator $\gamma: \mathbb{E}_1(J) \rightarrow h^{\alpha}(\Tone)$, defined by $\gamma v := v(0)$, is well-defined and we denote by $\gamma \mathbb{E}_1$ the image of $\gamma$ in $h^{\alpha}(\Tone)$, which is a Banach space when equipped with the norm
\[
\| f \|_{\gamma \mathbb{E}_1} := \inf \Big\{ \|v\|_{\mathbb{E}_1(J)}: v \in \mathbb{E}_1(J) \, \text{and} \, \gamma v = f \Big\}.
\]
By \cite[Theorem III.2.3.1]{AM95} and Proposition~\ref{Thm:hAlphaCAlphaClosure}(b) we see that 
\[
\begin{split}
\gamma \mathbb{E}_1 = (h^{\alpha}(\Tone), h^{2m + \alpha}(\Tone))_{\mu} = h^{2m \mu + \alpha}(\Tone), \qquad &\mu \in (0,1)\\
\gamma \mathbb{E}_1 := h^{2m + \alpha}(\Tone)\qquad &\mu = 1,
\end{split}
\]
where $(\cdot,\cdot)_{\eta}$ denotes the continuous interpolation functor of DaPrato and Grisvard, c.f. \cite{AM95,LUN95}, and the interpolation space characterization holds when $2m \mu + \alpha \notin \mathbb{Z}$ (up to equivalent norms).

For $B \in \mathcal{L}(h^{2m + \alpha}(\Tone),h^{\alpha}(\Tone)),$ closed on $h^{\alpha}(\Tone)$, we say that $\big( \mathbb{E}_0(J), \mathbb{E}_1(J) \big)$ is a \emph{pair of (continuous) maximal regularity} for $B$ %, and write $B \in \mathcal{M}_{\mu}(\alpha,J)$, 
if 
\[
\left( \frac{d}{dt} + B, \, \gamma \right) \in \mathcal{L}_{isom}(\mathbb{E}_1(J), \mathbb{E}_0(J) \times \gamma \mathbb{E}_1),
\]
$\mu \in (0,1], \, \alpha \in \mathbb{R}_+ \setminus \mathbb{Z}$ and $J = [0,T]$ for some $T > 0$.
In particular, we see that $\big( \mathbb{E}_0(J), \mathbb{E}_1(J) \big)$ is a pair of maximal regularity for $B$ if and only if for every $(f, u_0) \in \mathbb{E}_0(J) \times \gamma \mathbb{E}_1$, there exists a unique solution $u \in \mathbb{E}_1(J)$ to the inhomogeneous Cauchy problem with operator $B$.

%%%%%%%%%%%%%%%%%%%%% DaPrato-Grisvard-Angenent Max Reg %%%%%%%%%%%%%%%%%%%%%%%%%%%%%%

\subsection{Maximal Regularity and Generation of Analytic Semigroups}

Our goal is to show that $\big( \mathbb{E}_0(J), \mathbb{E}_1(J) \big)$ is a pair of maximal regularity for $\mathcal{A}$ for arbitrary $\alpha \in \mathbb{R}_+ \setminus \mathbb{Z}$ and $J=[0,T]$, given minimal regularity assumptions on the coefficients $b_k$. In particular, fix $m \in \mathbb{N}$ and coefficients $b_k \in h^{\alpha}(\Tone)$, $k = 0, \ldots, 2m$ such that $\mathcal{A} := \mathcal{A}(\cdot,D) := \sum_{k \leq 2m} b_k(\cdot) D^k$ satisfies the \emph{uniform ellipticity} conditions \eqref{Eqn:UniformEllipticity}. The tool we are going to use to prove this maximal regularity result is the following theorem of DaPrato, Grisvard and Angenent, which was originally proved by DaPrato and Grisvard \cite{DPG79} in the case $\mu = 1$ and then generalized to $\mu \in (0,1)$ by Angenent \cite{AN90}.

\begin{thm}[DaPrato-Grisvard-Angenent]\label{Thm:DaPratoGrisvard}
Fix $\eta \in (0,1)$, $\mu \in (0,1]$ and $J := [0,T]$ for $T > 0$. Suppose that $(E_0,E_1)$ is a pair of densely embedded Banach spaces and consider an operator $A \in \mathcal{H}(E_1,E_0)$. Now, set 
\begin{equation*}
\begin{split}
E_2 := E_2(A) := (D(A^2), \| \cdot \|_2) \quad \text{equipped with the norm} \quad \| \cdot \|_2 := \| A \cdot \|_1 + \| \cdot \|_1,\\
\begin{aligned}
&E_{\eta} := (E_0, E_1)_{\eta}, \qquad \qquad E_{1 + \eta} := E_{1 + \eta}(A) := (E_1, E_2(A))_{\eta}, \qquad \qquad\\
&\qquad A_{\eta} := \text{ the maximal $E_{\eta}$-realization of $A$.}
\end{aligned}
\end{split}
\end{equation*}
It follows that
\[
\big( \mathbb{E}_{\eta}(J), \mathbb{E}_{1 + \eta}(J) \big) := \Big( BUC_{1 - \mu}(J,E_{\eta}), BUC_{1 - \mu}^1(J, E_{\eta}) \cap BUC_{1 - \mu}(J, E_{1 + \eta}) \Big),
\]
is a pair of maximal regularity for $A_{\eta}$.
\end{thm}
It is also a well-known result that $A_{\eta} \in \mathcal{H}(E_{1 + \eta}, E_{\eta}),$ c.f. \cite[Section III.3.2]{AM95}.

Due to the continuous interpolation spaces constructed in the theorem, we see that we cannot directly derive maximal regularity results for $\mathcal{A}$ in $h^{\alpha}(\Tone)$. In particular, when applying Theorem~\ref{Thm:DaPratoGrisvard}, the derived maximal regularity results are necessarily in a little-H$\ddo$lder space with slightly larger exponent than where we assume analytic semigroup generation results. Moreover, it is in general quite difficult to characterize the operator-dependent space $E_2(\mathcal{A})$, which is in turn dependent upon the regularity conditions imposed on the coefficients $b_k$. However, we are able to take advantage of flexibility in Theorem~\ref{FullResult}, with respect to the regularity exponents, in order to work around these difficulties and prove the following result.

\begin{thm}\label{Thm:ExistenceAndUniqueness}
Fix $\alpha \in \mathbb{R}_+ \setminus \mathbb{Z}$, $m \in \mathbb{N}$, $\mu \in (0,1]$ and $J = [0,T]$, for $T > 0$ arbitrary. Suppose the operator $\mathcal{A} := \mathcal{A}(\cdot,D) = \sum_{k \leq 2m} b_k(\cdot) D^k$, with coefficients $b_k \in h^{\alpha}(\Tone)$ is uniformly elliptic, as in \eqref{Eqn:UniformEllipticity}. Then
\[
\left( \frac{d}{dt} + \mathcal{A}, \, \gamma \right) \in \mathcal{L}_{isom}(\mathbb{E}_1(J), \mathbb{E}_0(J) \times \gamma \mathbb{E}_1).
\]
In particular, given any pair $(f, u_0) \in BU\!C_{1 - \mu}(J, h^{\alpha}(\Tone)) \times \gamma \mathbb{E}_1$, there exists a unique solution $u \in BU\!C_{1 - \mu}^1 (J, h^{\alpha}(\Tone)) \cap BU\!C_{1 - \mu}(J, h^{2m + \alpha}(\Tone))$ to the inhomogeneous Cauchy problem \eqref{Eqn:CP2} on $J$.
\end{thm}
 
\begin{proof}
Fix $\beta \in \mathbb{R}_+ \setminus \mathbb{Z}$ such that $\beta < \alpha < 2m + \beta$ and fix $\eta := \frac{\alpha - \beta}{2m}$, then we see that $\eta \in (0,1)$ and $2m \eta + \beta = \alpha$. $\mathcal{A}$ is trivially realized as an operator from $h^{2m + \beta}(\Tone)$ to $h^{\beta}(\Tone)$ by Proposition~\ref{Thm:hAlphaCAlphaClosure}(a), 
%(i.e. we view $\mathcal{A}$ as an operator on $h^{\beta}(\Tone)$, so that the maximal $h^{\alpha}(\Tone)$-realization of $\mathcal{A}$ is simply itself)
so that, by Theorem~\ref{FullResult}, we know $\mathcal{A} \in \mathcal{H}(h^{2m + \beta}(\Tone),h^{\beta}(\Tone))$. Now we construct the spaces $E_2, E_{\eta}$ and $E_{1 + \eta}$ as in Theorem~\ref{Thm:DaPratoGrisvard}, and we apply Proposition~\ref{Thm:hAlphaCAlphaClosure}(b) when possible. Namely, we set
\[
E_2 := \big\{ f \in h^{\beta}(\Tone): \mathcal{A}f \in h^{2m + \beta}(\Tone) \big\},
\]
equipped with the graph norm $\| \cdot \|_2 := \|\mathcal{A} \cdot \|_{h^{2m + \beta}} + \| \cdot \|_{h^{2m + \beta}}$,
\begin{align*}
E_{\eta} &:= (h^{\beta}(\Tone), h^{2m + \beta}(\Tone))_{\eta} = h^{\alpha}(\Tone)\\
E_{1 + \eta} &:= (h^{2m + \beta}(\Tone), E_2)_{\eta};
\end{align*}
notice that, a priori, we cannot conclude $E_{1 + \eta}$ coincides with a little-H$\ddo$lder space without a proper characterization of $E_2$. However, by uniform ellipticity of $\mathcal{A}$, with coefficients $b_k$ in $h^{\alpha}(\Tone)$, we know that $\mathcal{A} \in \mathcal{H}(h^{2m + \alpha}(\Tone), h^{\alpha}(\Tone)),$ by Theorem~\ref{FullResult} again. Meanwhile, by the remark following Theorem~\ref{Thm:DaPratoGrisvard}, we see that $\mathcal{A} \in \mathcal{H}(E_{1 + \eta}, h^{\alpha}(\Tone))$. Hence, we can find $\omega > 0$ sufficiently large so that 
\[
(\omega + \mathcal{A}) \in \mathcal{L}_{isom}(h^{2m + \alpha}(\Tone), h^{\alpha}(\Tone)) \cap \mathcal{L}_{isom}(E_{1 + \eta}, h^{\alpha}(\Tone)).
\]
However, it follows that $(\omega + \mathcal{A})^{-1} \circ (\omega + \mathcal{A}): h^{2m + \alpha}(\Tone) \rightarrow E_{1 + \eta}$ is an isometric isomorphism, by commutativity. So, $h^{2m + \alpha}(\Tone)$ and $E_{1 + \eta}$ coincide, up to equivalent norms, and
%(Now use a commutative diagram, or start with $u \in D(A_{\eta})$ and show that $u \in h^{2m + \alpha}(\Tone)$ which then implies that $D(A_{\eta}) = E_{1 + \eta}$ coincides with $h^{2m + \alpha}(\Tone)$, up to equivalent norms...) Hence, we conclude that $E_{1 + \eta} = h^{2m + \alpha}(\Tone)$ by maximality of the domain of definition for a strongly continuous semigroup. Then $\mathcal{A} \in \mathcal{M}_{\mu}(\alpha,J)$ by Theorem~\ref{Thm:DaPratoGrisvard}, and the claim follows by definition of maximal regularity.
it follows that $\big( \mathbb{E}_0(J), \mathbb{E}_1(J) \big)$ is a pair of (continuous) maximal regularity for $\mathcal{A}$, by Theorem~\ref{Thm:DaPratoGrisvard}, which demonstrates the claim.
\end{proof}

%% file: Vector-ValueCase-J3E.tex
\section{Vector-Valued Setting}\label{Section:VectorSetting}

For the remainder of the paper, let $E = (E, | \cdot |)$ denote an arbitrary (non-trivial) Banach space
over $\mathbb{C}$. Again, consider the inhomogeneous problem \eqref{Eqn:CP} with periodicity enforced.
However, suppose that one is given vector-valued functions, $u_0, f(t, \cdot): \Tone \rightarrow E$, and operator-valued
coefficients, $b_k : \Tone \rightarrow \mathcal{L}(E)$. It turns out that, with only minor modifications and appropriate 
alterations to definitions, the preceding results continue to hold in this more general setting. 
In this section, we highlight the necessary changes to the preceding theory and state results in this vector-valued setting.

\subsection{Vector-Valued Function Spaces}

Following common conventions, we denote by $C(\Tone, E)$, $C^{\theta}(\Tone, E)$, and $h^{\theta}(\Tone, E)$, 
the classes of regular $E$-valued functions analogous to the scalar-valued cases defined in Section~\ref{SubSec:Regularity},
the definitions of which remain essentially unchanged.
Moreover, one will note that Proposition~\ref{Thm:hAlphaCAlphaClosure} is a simplified version
of \cite[Proposition 0.2.1 and Theorem 1.2.17]{LUN95}, which were already stated in the vector-valued setting, so there is no
trouble in getting these same results for $E$-valued functions. In order to give an adequate definition of $E$-valued Besov spaces, 
however, one will need the concept of vector-valued distributions. 

Taking $\mathcal{D}(\Tone)$ to be the smooth $\mathbb{C}$-valued functions over $\Tone$, as before,
we define the space of $E$-valued distributions $\mathcal{D}'(\Tone, E) := \mathcal{L}(\mathcal{D}(\Tone), E)$ and 
we equip $\mathcal{D}'(\Tone,E)$ with the weak-star topology over $\mathcal{D}(\Tone)$. One can see that the same 
definitions of Fourier coefficients and results on Fourier series representations continue to hold, c.f. \cite{AB04}.
In particular, one can see that, for every $f \in \mathcal{D}'(\Tone, E)$, it holds that
\[
f = \sum_{k \in \mathbb{Z}} e_k \otimes \hat{f}(k) \qquad \text{(convergence in $\mathcal{D}'(\Tone, E)$)},
\]
where $e_k \in \mathcal{D}(\Tone)$ has the same definition as before and $e_k \otimes y$ denotes the function
$[x \mapsto e^{ikx}y]: \Tone \rightarrow E$ for $y \in E$ given.
Then, we define the $E$-valued periodic Besov spaces $B^s_{p,q}(\Tone, E)$ as before,
by making use of collections of dyadic decompositions $\Phi(\mathbb{R})$, and we derive analogous results to those discussed
in the scalar setting. 

\subsection{Operator-Valued Fourier Multipliers}

Now, with vector-valued Besov\\ spaces established, we consider Fourier multiplier results in this setting. As discussed in \cite{AB04}, 
the \emph{Fourier type} of the underlying Banach space $E$ will affect the statement of the Fourier multiplier result. To be clear,
we are given a sequence $\left( M_k \right)_k \subset \mathcal{L}(E)$ and consider the associated (formal) operator
\[
T : \sum_{k \in \mathbb{Z}} e_k \otimes \hat{f}(k) \longmapsto \sum_{k \in \mathbb{Z}} e_k \otimes M_k \hat{f}(k).
\]
The following multiplier theorem will work for the general case where $E$ is a Banach space with arbitrary Fourier type. 
We note that Lemma~\ref{Lem:FourierTypeBound} does not hold in this general case.

\begin{thm}\label{Thm:BanachFourierMultipliers}
Let $r, s \in \mathbb{R}_+$ and $1 \leq p, q \leq \infty$. Suppose that $(M_k)_{k \in \mathbb{Z}} \subset \mathcal{L}(E)$ is a sequence such that
\begin{equation}\label{Eqn:BanachMultiplierAssumptions}
\begin{split}
s_1 &:= \sup_{k \in \mathbb{Z} \setminus \{ 0 \}} |k|^{r - s} \| M_k \| < \infty,\\
s_2 &:= \sup_{k \in \mathbb{Z} \setminus \{ 0 \}} |k|^{r - s + 1} \| M_{k + 1} - M_k \| < \infty,\\
s_3 &:= \sup_{k \in \mathbb{Z} \setminus \{ 0 \}} |k|^{r - s + 2} \| M_{k + 1} - 2 M_k + M_{k - 1} \| < \infty.
\end{split}
\end{equation}
Then the Fourier multiplier with symbol $( M_k )_{k \in \mathbb{Z}}$ is a continuous mapping from $B^{s}_{p,q}(\Tone, E)$ to $B^{r}_{p,q}(\Tone, E)$, namely
\[
T: \left[ \sum_{k \in \mathbb{Z}} e_k \otimes \hat{f}(k) \longmapsto \sum_{k \in \mathbb{Z}} e_k \otimes M_k \hat{f}(k) \right] \in \mathcal{L} \left( B^{s}_{p,q}(\Tone, E), B^{r}_{p,q}(\Tone, E) \right).
\]
\end{thm}

A proof of Theorem~\ref{Thm:BanachFourierMultipliers} follows from \cite[Theorem 2.2.1]{MAT09}, by restating the proof in the $E$-valued setting.
On the other hand, if we find that the Fourier type of $E$ is in the interval $(1,2]$, then Lemma~\ref{Lem:FourierTypeBound} is known to hold
and we have the analogous statement to Theorem~\ref{Thm:FourierMultipliers}, without the necessity of checking the term $s_3$. Note that
this sharper case includes the situation $E$ a Hilbert space, where the Fourier type is exactly $2$.

\subsection{Ellipticity With Operator-Valued Coefficients}

%Following the methods employed in the scalar-valued setting, we will proceed by giving an appropriate definition for ellipticity
%of differential operators with $\mathcal{L}(E)$-valued coefficients and proving a semigroup generation result for an elliptic operator with
%constant coefficients. 
Now, fix a collection $\{ b_k : k = 0, \ldots, 2m \} \subset h^{\alpha}(\Tone, \mathcal{L}(E))$ of operator-valued coefficient functions and consider the differential operator $\mathcal{A}$, acting on $h^{2m + \alpha}(\Tone, E)$, defined by
\[
\mathcal{A}u(x) := \mathcal{A}(x,D) u(x) := \sum_{k = 0}^{2m} b_k(x) \, (D^k u)(x) = \sum_{k = 0}^{2m} i^k \, b_k(x) \, u^{(k)}(x), \qquad x \in \Tone.
\]
By the embedding property Proposition~\ref{Thm:hAlphaCAlphaClosure}(a) and the regularity assumptions on $b_k$ and $u$, it follows immediately that $\mathcal{A}$ maps $h^{2m + \alpha}(\Tone)$ into $h^{\alpha}(\Tone)$. Now, denote by $\sigma \mathcal{A}: \Tone \times \mathbb{R} \rightarrow \mathcal{L}(E)$ the \emph{principal symbol} of $\mathcal{A}$, defined by $\sigma \mathcal{A}(x,\xi) := \xi^{2m} b_{2m}(x).$ We say that 
$\mathcal{A}$ is a \emph{normally elliptic} operator on $\Tone$ if there exist constants $c_1 \ge 1$ and $\theta \in (\pi / 2, \pi)$ so that the properties
\begin{equation}\label{Eqn:NormallyElliptic}
\begin{split}
\rho(- \sigma \mathcal{A}(x, \xi)) &\supset \Sigma_{\theta} := \{ z \in \mathbb{C}: | \arg z | \le \theta \} \cup \{ 0 \}\\
(1 + |\lambda|)& \| (\lambda + \sigma \mathcal{A}(x,\xi))^{-1} \| \le c_1, \quad \lambda \in \Sigma_{\theta},
\end{split}
\end{equation}
hold for all $x \in \Tone$ and $|\xi| = 1$. This definition coincides with the definition of normally elliptic operators presented in \cite[Section 3]{AM01} and one will note that this definition generalizes the notion of uniform ellipticity, as in \eqref{Eqn:UniformEllipticity}.
Moreover, as mentioned by Amann in \cite{AM01}, in the case that $E$ is finite-dimensional, this definition of normal ellipticity is equivalent to the condition that there exist $0 < r < R$ such that 
\[
\sigma(\sigma \mathcal{A}(x,\xi)) \subset \{z \in \mathbb{C}: \re \, z \ge r \} \cap \{ z \in \mathbb{C}: |z| \le R \}, \qquad \text{for } x \in \Tone, \; |\xi| = 1.
\]
%for all $(x,\xi) \in \Tone \times \mathbb{R}$ with $|\xi| = 1$.
%, and we get to the first semigroup generation result in the $E$-valued setting.

\begin{thm}\label{Thm:BanachGenerationWithConstant}
Let $E$ be a Banach space, $m \in \mathbb{N}$, $\alpha \in \mathbb{R}_+ \setminus \mathbb{Z}$ and consider the differential operator $\mathcal{A}_b := b \,  D^{2m}$ with constant coefficient $b \in \mathcal{L}(E)$. If $\mathcal{A}_b$ is normally elliptic, with constant $c_1 > 0$,
and $c_2 \geq c_1 > 0$ is chosen so that $\| b \| \leq c_2$, then $- \mathcal{A}_b$ generates a (strongly continuous) analytic semigroup on $h^{\alpha}(\Tone, E)$ with domain $h^{2m + \alpha}(\Tone, E)$. Moreover, for any $\omega > 0$, there exists $\kappa = \kappa(\omega, c_1, c_2, m)$ such that 
\[
\mathcal{A}_b \in \mathcal{H} \big(h^{2m + \alpha},h^{\alpha},\kappa(\omega, c_1, c_2, m), \omega \big).
\]
\end{thm}

\begin{proof}
The proof of this result follows the same method used to prove Theorem~\ref{Thm:GenerationWithConstant},
however, in this vector-valued setting, we must derive bounds for the term $s_3$ before applying our Fourier multiplier result,
Theorem~\ref{Thm:BanachFourierMultipliers} in this case.
Fix $\alpha \in \mathbb{R}_+ \setminus \mathbb{Z}$, $\omega > 0$ and $b \in \mathcal{L}(E)$ as indicated.
Notice that $-\mathcal{A}_b$ is now associated with the operator-valued multiplier symbol
$\left( M_k \right)_k := \left( -k^{2m} b \right)_k \subset \mathcal{L}(E)$.
Now, we can make formally identical claims to those stated in the scalar-valued setting.

\noindent {\bf Claim 1:} \emph{ $(\lambda + \mathcal{A}_b) \in \mathcal{L}_{isom} (h^{2m + \alpha}(\Tone, E), h^{\alpha}(\Tone, E))$ for $\re \, \lambda \geq \omega$, i.e. 
\[
\rho(-\mathcal{A}_b) \supset \{ \lambda \in \mathbb{C}: \re \, \lambda \geq \omega \}.
\]
Moreover, the set $\{ \| (\lambda + \mathcal{A}_b)^{-1} \|_{\mathcal{L}(h^{\alpha},h^{2m + \alpha})}: \re \, \lambda \geq \omega \}$ is uniformly bounded by some $M_1 = M_1(\omega, c_1, c_2, m) < \infty$.}

Let $\lambda$ be fixed with $\re \, \lambda \ge \omega$. 
The fact that $\mathcal{A}_b \in \mathcal{L}(C^{2m + \sigma}(\Tone), C^{\sigma}(\Tone))$ follows in the same way as the scalar case
with $\| \lambda + \mathcal{A}_b \| \le (c(\sigma)|\lambda| + c_2)$. 
Now consider the symbol $\left( \tilde{M}_k(\lambda) \right)_k := \Big( (\lambda + k^{2m}b)^{-1} \Big)_k$,
where the condition of normal ellipticity guarantees that $\re \, \lambda \ge 0$ is sufficient to see that
$\lambda \in \rho(\sigma\mathcal{A}_b(x,\xi))$. Moreover, in the constant coefficient case, it follows
that $\sigma\mathcal{A}_b(x,\xi) \equiv b$ for $|\xi|=1$. Now, the second condition of normal ellipticity, \eqref{Eqn:NormallyElliptic},
gives adequate flexibility to see that $\frac{\lambda}{k^{2m}} \in \rho(b)$ and we conclude that 
$\tilde{M}_k(\lambda) \in \mathcal{L}(E)$ is well-defined, for $k \in \mathbb{Z}$. Further, notice that
\[
\tilde{M}_k(\lambda) := \left( \lambda + k^{2m}b \right)^{-1} = k^{-2m} \left( \frac{\lambda}{k^{2m}} + b \right)^{-1} \qquad k \not= 0,
\]
and we make use of the second expression in verifying the conditions of the Fourier multiplier theorem.

Using the resolvent bounds given in the normal ellipticity definition, we see that, concerning the symbol
$\left( \tilde{M}_k(\lambda) \right)_k$, we have
\[
s_1 \le c_1 < \infty, \qquad
s_2 \le \left( \frac{c_1 c_2}{\omega} \vee c_2 \right) \sup_{k \in \mathbb{Z} \setminus \{ -1 \}} \left( \frac{|k||(k + 1)^{2m} - k^{2m}|}{|k+1|^{2m}} \right) < \infty.
\]
Meanwhile, notice that, for $k \not= \pm 1$, 
\begin{align*}
|k&|^{2m + 2} \left\| (\lambda + (k + 1)^{2m} b)^{-1} - 2(\lambda + k^{2m}b)^{-1} + (\lambda + (k - 1)^{-1}b)^{-1} \right\|\\
&\le \left\| \left( \frac{\lambda}{(k + 1)^{2m}} + b \right)^{-1} \right\| 
		 \left\| \left( \frac{\lambda}{k^{2m}} + b \right)^{-1} \right\| 
		 \left\| \left( \frac{\lambda}{(k - 1)^{2m}} + b \right)^{-1} \right\|\\
&\hspace{1.5em}  	 \left( \frac{|k|^2}{|k + 1|^{2m} |k - 1|^{2m}} \right)
		 \bigg[ \left\| -\lambda \left( (k + 1)^{2m} - 2k^{2m} + (k - 1)^{2m} \right) b \right\|\\  
&\hspace{2.5em} + \left\| \left( (k + 1)^{2m} (k^{2m} - (k - 1)^{2m}) + (k - 1)^{2m} ( (k + 1)^{2m} - k^{2m}) \right) b^2 \right\| \bigg]\\
&\le \frac{c_1^3 c_2}{(1 + | \lambda |)^3} \Big( |\lambda| + c_2 \Big) \mathcal{K}_1(k) \le c_1^3 c_2 (1 + c_2) \mathcal{K}_1(k),
\end{align*}
where $\mathcal{K}_1$ is a bounded function in $k$. Similarly, in case $k = \pm 1$, we see that 
\[
\left\| (\lambda)^{-1} - 2(\lambda + b)^{-1} + (\lambda + 2^{2m}b)^{-1} \right\| \le 2^{2m} c_1^2 c_2 (1 + c_2).
\]
Hence, it follows that $s_3 < \infty$, and bounded by terms which only depend upon $\omega, m,  c_1,$ and $c_2$ and we 
can apply Theorem~\ref{Thm:BanachFourierMultipliers} to prove the claim. We again see that the operator $R(\lambda)$ 
associated with the symbol $\left( \tilde{M}_k(\lambda) \right)_k$ coincides with the inverse of $(\lambda + \mathcal{A}_b)$.

\noindent {\bf Claim 2:} \emph{ $\lambda (\lambda + \mathcal{A}_b)^{-1} \in \mathcal{L}(h^{\alpha}(\Tone, E))$ for $\re \, \lambda \geq \omega,$ and there is an upper bound $M_2 = M_2(\omega, c_1, c_2, m) < \infty$ for the set $\left\{ | \lambda | \|(\lambda + \mathcal{A}_b)^{-1} \|_{\mathcal{L}(h^{\alpha})} : \re \, \lambda \geq \omega \right\}$.}

This claim is verified by applying the same techniques as above to the symbol $\left( \lambda (\lambda + k^{2m}b)^{-1} \right)_k$.
Working through the details, one verifies that the $s_i$ terms, $i = 1, 2, 3,$ are bounded exactly the same as in Claim 1 above.
Hence, the desired result holds, and the proof of the theorem proceeds exactly as in the scalar-valued setting.

\end{proof}

\subsection{Semigroup Generation and Maximal Regularity}

We conclude the paper with statements of our main results in the setting of vector-valued functions. Their proofs are obtained
by direct application of the methods employed in the scalar-valued setting, with only minor changes of notation and definitions,
which have already been addressed in the preceding parts of this section.

\begin{thm}\label{BanachFullResult}
Let $E$ be a Banach space, $m \in \mathbb{N}, \; \alpha \in \mathbb{R}_+ \setminus \mathbb{Z}, b_k \in h^{\alpha}(\Tone, \mathcal{L}(E)),$ for $k = 0, \ldots, 2m$, and suppose the operator $\mathcal{A} := \mathcal{A}(\cdot,D) := \sum_{k = 0}^{2m} b_k(\cdot) D^k$ is normally elliptic. Then
\[
\mathcal{A} \in \mathcal{H} ( h^{2m + \alpha}(\Tone, E), h^{\alpha}(\Tone, E)).
\]
\end{thm}

\noindent Now, fix $\mu \in (0,1]$ and define the spaces 
\[
\begin{split}
&\mathbb{E}_0(J) := BUC_{1 - \mu}(J,h^{\alpha}(\Tone, E)), \qquad \alpha \in \mathbb{R}_+ \setminus \mathbb{Z}\\
&\mathbb{E}_1(J) := BUC^1_{1 - \mu}(J,h^{\alpha}(\Tone, E)) \cap BUC_{1-\mu}(J,h^{2m + \alpha}(\Tone, E)).
\end{split}
\] 
Then we get the maximal regularity result.

\begin{thm}\label{Thm:BanachExistenceAndUniqueness}
Fix a Banach space $E$, $\alpha \in \mathbb{R}_+ \setminus \mathbb{Z}$, $m \in \mathbb{N}$, $\mu \in (0,1]$ and $J = [0,T]$, for $T > 0$ arbitrary. Suppose the operator $\mathcal{A} := \mathcal{A}(\cdot,D) = \sum_{k \leq 2m} b_k(\cdot) D^k$, with coefficients $b_k \in h^{\alpha}(\Tone, \mathcal{L}(E))$ is normally elliptic, as in \eqref{Eqn:NormallyElliptic}. Then
\[
\left( \frac{d}{dt} + \mathcal{A}, \, \gamma \right) \in \mathcal{L}_{isom}(\mathbb{E}_1(J), \mathbb{E}_0(J) \times \gamma \mathbb{E}_1).
\]
In particular, given any pair $(f, u_0) \in BU\!C_{1 - \mu}(J, h^{\alpha}(\Tone, E)) \times \gamma \mathbb{E}_1$, there exists a unique solution $u \in BU\!C_{1 - \mu}^1 (J, h^{\alpha}(\Tone, E)) \cap BU\!C_{1 - \mu}(J, h^{2m + \alpha}(\Tone, E))$ to the inhomogeneous Cauchy problem \eqref{Eqn:CP} on $J$.
\end{thm}

%%%%%%%%%%%%%%%%%%%%%%%%%%%% ACKNOWLEDGEMENTS %%%%%%%%%%%%%%%%%%%%%%%%%%%%%%%
\section*{Acknowledgements}

I would like to express my gratitude to my advisor, Prof. Gieri Simonett,
who has provided many helpful discussions, encouraging words and a critical eye
during the preparation and revisions of this paper. 
I would also like to thank Mathias Wilke and 
Jan Pr$\ddot{\rm u}$ss for their input and advice in the preparation of this paper.
Thank you also to the anonymous referees for helpful suggestions.